\documentclass[11pt, oneside]{article}  	
\usepackage{geometry}               
\usepackage{graphicx}				
\usepackage{subfigure}
\usepackage{epsfig}
\usepackage{amsthm,amsmath,amssymb,graphicx,mathrsfs,braket,latexsym}
\usepackage[monochrome]{color}

\newtheorem{thm}{Theorem}[section]
\newtheorem{prop}[thm]{Proposition}
\newtheorem{lem}[thm]{Lemma}
\newtheorem{cor}[thm]{Corollary}
\newtheorem{setting}{Setting}
\theoremstyle{definition}
\newtheorem{defi}[thm]{Definition}
\newtheorem{rem}[thm]{Remark}
\newtheorem{ex}[thm]{Example}

\newtheorem{ques}[thm]{Question}
\newtheorem{prob}[thm]{Problem}

\DeclareMathOperator{\supp}{supp}
\DeclareMathOperator{\card}{card}

\DeclareMathOperator{\spec}{sp}
\DeclareMathOperator{\spa}{span}

\DeclareMathOperator{\Int}{int}
\DeclareMathOperator{\vol}{vol}

\newcommand{\Zpo}{\mathbb{Z}_{>0}}

\newcommand{\calT}{\mathcal{T}}
\newcommand{\calA}{\mathcal{A}}

\newcommand{\calS}{\mathcal{S}}
\newcommand{\calP}{\mathcal{P}}
\newcommand{\calQ}{\mathcal{Q}}
\newcommand{\calR}{\mathcal{R}}
\newcommand{\e}{\varepsilon}
\newcommand{\Q}{\mathbb{Q}}
\newcommand{\R}{\mathbb{R}}
\newcommand{\C}{\mathbb{C}}
\newcommand{\Z}{\mathbb{Z}}

\newcommand{\sci}{\wedge}

\newcommand{\Tk}{{\mathcal T}}
\newcommand{\N}{\mathbb{N}}
\newcommand{\Lam}{\Lambda}

\newcommand{\dens}{\mbox{\rm dens}}
\newcommand{\Gk}{{\mathcal G}}
\newcommand{\Ok}{{\mathcal O}}
\newcommand{\om}{\omega}

\definecolor{Yasushi}{rgb}{0.3,0.6,0.1}

\title{On arithmetic progressions in non-periodic self-affine tilings}
\author{Yasushi Nagai
\thanks{School of Mathematics and Statistics, Faculty of Science, Technology, Engineering and Mathematics, The Open University, Walton Hall, Milton Keynes, MK7 6AA, UK,
e-mail: yasushi.nagai@open.ac.uk}
\and Shigeki Akiyama 
\thanks{Institute of Mathematics, University of Tsukuba, 1-1-1 Tennodai, Tsukuba, Ibaraki 305-8571, Japan,
e-mail:akiyama@math.tsukuba.ac.jp}
\and 
Jeong-Yup Lee
\thanks{Department of Mathematics Education, Catholic Kwandong University, Gangneung, Gangwon 210-701, Korea, or KIAS, 85 Hoegiro, Dongdaemun-gu, Seoul 02455, Korea,
e-mail:jylee@cku.ac.kr}}
\date{\today}							

\begin{document}

\maketitle

\begin{abstract}
         We study the repetition of patches in self-affine tilings in $\R^d$.
          In particular, we study the
         existence and non-existence of arithmetic progressions.
         We first show that an arithmetic condition of the expansion map 
         for a self-affine tiling implies the
          non-existence of certain one-dimensional arithmetic progressions.
          Next, we show that the existence of full-rank infinite arithmetic progressions,
          pure discrete dynamical spectrum, and limit periodicity are all equivalent
          for a certain class of self-affine tilings.
          We finish by giving a complete picture 
          for the existence/non-existence of
          full-rank infinite arithmetic progressions in the self-similar tilings in $\R^d$.
\end{abstract}

\section{Introduction}
A tiling is a cover of the Euclidean space $\R^{d}$ by a set of tiles without interior overlaps.
The simplest class of tilings is the one of crystallographic tilings, where a tiling $\calT$ is
crystallographic if its symmetry group is a crystallographic group, i.e. it has 
translational symmetry $\calT+x=\calT$ for the vectors $x$ in a basis of $\R^{d}$.
Although the crystallographic tilings are interesting, the discovery of quasicrystals requires us to
go beyond that category and study non-periodic tilings, that is, tilings without translational symmetry.
As models of quasicrystals, the non-periodic tilings that are ``ordered'' are important, although there
are several interpretations of the term ``ordered''.
The most important interpretation is that the tiling is pure point diffractive, which by definition
means that the diffraction measure is pure point.
(The diffraction measure models physical diffraction pattern.
The presence of point masses is an indication of
``order'' for the tiling.
For an introduction of diffraction spectrum, see for example \cite[chapter 9]{Baake-Grimm1}.)
 Being pure point diffractive is equivalent to
having pure discrete dynamical spectrum \cite{BL}. (We will define pure discrete dynamical
spectrum in page \pageref{def_pure_discrete_spec}.)
Being pure point diffractive is also equivalent to
being almost periodic  in the sense of Gou\'er\'e \cite{G} and to being mean almost periodic
\cite{LSS}.
 These forms of almost periodicity 
 characterize how patterns (patches) distribute in the tiling.
This is why the distribution of patches in a given tiling, especially the repetition of finite patches,
is important.

Another interpretation of the term ``ordered'' is that the tiling is limit-periodic. 
Here, a tiling is said to be limit-periodic if,
except for a set of zero-density tiles, all tiles $T$ repeat crystallographically, which means
that
there is a lattice $L$ of $\R^{d}$ for each $T$ 
such that the all translates $T+x$, $x\in L$, are included in the tiling.
We usually further assume that the lattice $L$ which appears in this way belongs to a decreasing
sequence of lattices.
This is a geometric version of Toeplitz sequences \cite{JK}.
The period doubling tiling and the chair tilings are well-known examples of limit-periodic tiling.

In this paper, we deal with a type of the 
repetition of patches, that is, arithmetic progressions of patches,
and discuss its relations with pure point diffraction and limit-periodicity
for self-affine tilings.
A self-affine tiling is a tiling with an inflation-subdivision symmetry, which
means that there exists an expansive linear map $Q$ that makes all tiles larger and we can subdivide
them into the tiles of original size to obtain the original tiling. Illustrative 
examples are found in Figure \ref{Fig1}, where we show a chair tiling and a Robinson triangle
 tiling, respectively.
 The latter is essentially the same (MLD) as one of the famous Penrose tilings 
 \cite[Section 6.2]{Baake-Grimm1}.

\begin{figure}[h]
\begin{center}
\subfigure[a chair tiling]{
\includegraphics[width=5cm]{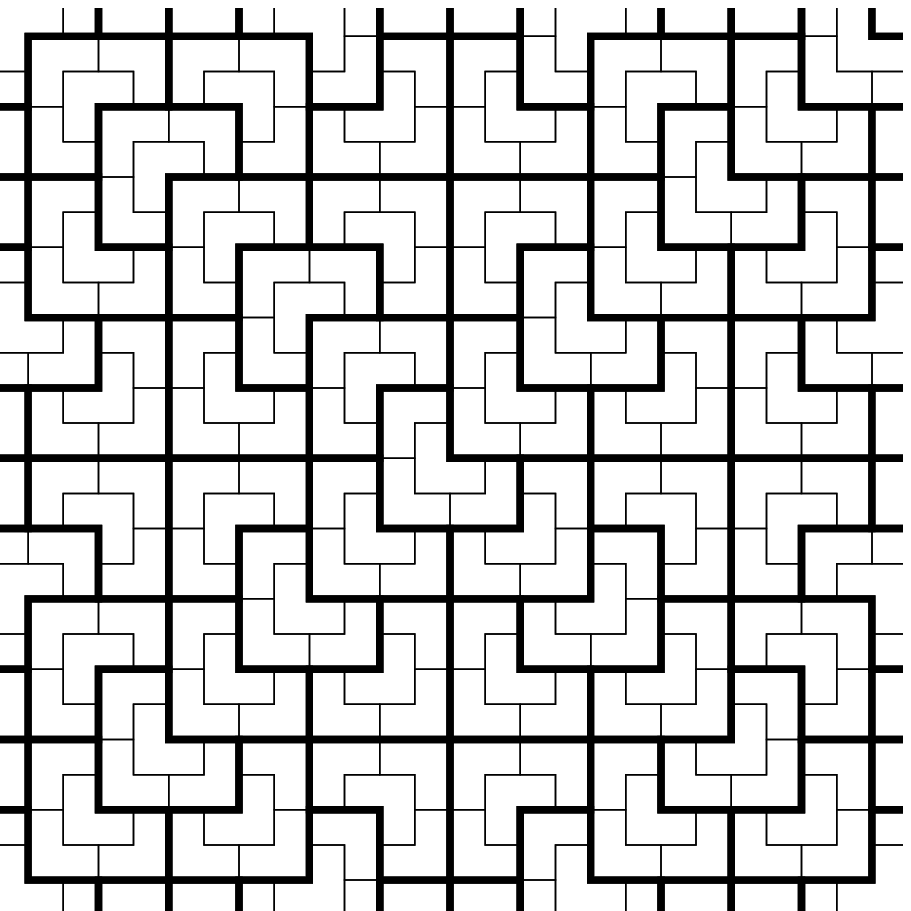}
\label{Chair}
}
\hspace{1cm}
\subfigure[a Robinson triangle tiling]{
\includegraphics[width=5cm]{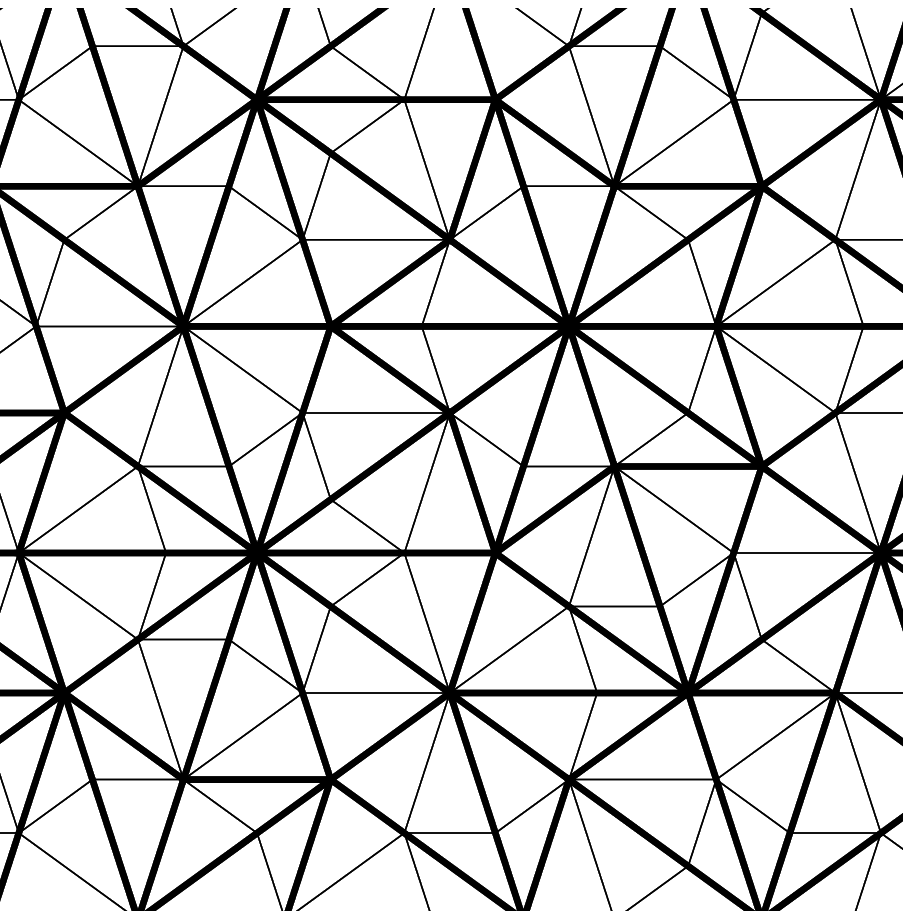}
\label{Penrose}
}
\caption{Self-affine tiling}
\label{Fig1}
\end{center}
\end{figure}

A one-dimensional arithmetic progression of a patch $\calP$ is a patch of the form
$\bigcup_{k=1}^{n}\calP+kx$, where $x$ is a non-zero vector of $\R^{d}$. Here,
$n$ is either a positive 
integer or $\infty$.
This is a pattern where the patch $\calP$ appears $n$ times,
separated in each instance by the translation vector $x$ from the previous.
It is interesting to observe 
that a chair tiling possess such an arithmetic progression in the diagonal direction,
as we see in Figure \ref{Fig1}.
In this article, we first give an arithmetic condition on the expansion map $Q$ for
self-affine tilings which implies the
\emph{non-existence} of arbitrarily long arithmetic progressions for fixed $\calP$ and $x$
(Theorem \ref{thm_using_conti_eigenft}, Theorem \ref{thm_internal_space}).
For example, in Penrose tilings, the arithmetic progressions stop at a certain finite $n$ which
depends on each $\calP$ and $x$.
This is done in Section \ref{section_non-existence}.

Next, in Section \ref{section_full-rank-AP},
we deal with \emph{full-rank infinite arithmetic progressions}, which are by definition
patches of the form
\begin{align*}
       \bigcup_{(l_{1},l_{2},\ldots l_{d})\in\mathbb{Z}^{d}}\calP+\sum_{j=1}^{d} l_{j}b_{j},
\end{align*}
where $\{b_{1},b_{2},\ldots b_{d}\}$ is a basis of $\R^{d}$.
Interestingly, for a class of self-affine tilings, the existence of full-rank arithmetic progressions is
equivalent to having pure discrete dynamical spectrum and also to being limit periodic
(Theorem \ref{full-rank-infiniteAP_implies_pp}, 
Corollary \ref{cor_sufficient_conditions_for_limit-period}).

In Section \ref{section_one-dim}, we give a complete picture for the existence/non-existence of
full-rank infinite arithmetic progressions in a self-similar tiling in $\R^d$. We then give further problems 
in Section \ref{sec_further_questions}
and
finish the article with Appendix, where some of the proofs are given.

All of this is done after introducing the necessary notation and recalling known results in 
Section \ref{section_notation_known-results}.

\section{Notation and known results}
\label{section_notation_known-results}
In this section, we recall relevant notation and known results.
In the whole article, $B(x,R)$ is the closed ball in the $d$-dimensional Euclidean space
$\mathbb{R}^{d}$ with its center $x\in\R^{d}$ and
radius $R>0$. Often $B(0,R)$ is denoted by $B_{R}$.
The symbol $\mathbb{T}$ denotes the one-dimensional torus: 
$\mathbb{T}=\{z\in\mathbb{C}\mid |z|=1\}$.

\subsection{Tilings and substitutions}

Let $L$ be a finite set. An $L$-labeled tile is a pair $(S,l)$ of a compact non-empty
subset $S$ of $\R^d$, such that $\overline{S^\circ}=S$
(the closure of the interior coincides with the original $S$), and an $l\in L$.
We often fix $L$ and call $L$-labeled tiles just tiles.
For a tile $T=(S,l)$, we write $S=\supp T$ and $l=l(T)$.
We also write $\Int(T)=(\supp T)^{\circ}$.
For a $T=(S,l)$ and an $x\in\R^{d}$, we set $T+x=(S+x,l)$.

Alternatively, we also 
call a compact non-empty subset $S$ of $\R^{d}$ such that $S=\overline{S^{\circ}}$
a tile. 
Thus ``a tile'' means either a ``labeled'' or an ``un-labeled'' tile.
Both types of
 tiles are useful in aperiodic order. For example, if we consider the geometric realization of
a constant-length symbolic substitution, we have to give labels to tiles in order to distinguish 
the intervals for different letters. On the other hand, for many geometric substitutions labels are not
necessary.
\footnote{Sometimes it is not a good idea to
 use $\{(A_{i},i)\mid i=1,2,\ldots n\}$ as the alphabet of a geometric substitution, 
since this does not describe the rotational symmetry of a substitution
($\omega(R(P))=R(\omega(P))$, where $P$ is a proto-tile and $R$ is a rotation), which we often
use tacitly. It is sometimes more  convenient to use unlabeled tiles and sometimes labeled ones,
and so it is useful to include both types of tiles in the definition of tiles.}
For an unlabeled tile $S$, we also use the notation $\supp S=S$ and $\Int(S)=S^{\circ}$.

A set $\calP$ of tiles is called a  \emph{patch} if for any $T_1,T_2\in\calP$ with
$\Int(T_1)\cap\Int(T_2)\neq\emptyset$, we have $T_1=T_2$.
For a patch $\calP$, its \emph{support} $\supp\calP$ is defined by
$\supp\calP=\overline{\bigcup_{T\in\calP}\supp T}$.
If  a patch $\calT$ satisfies the condition $\supp\calT=\R^d$, we call $\calT$ a \emph{tiling}.
Often patches are assumed to be finite sets, but in this article we do not assume that.
(Infinite patches  that are not tilings appear in this article.)
A patch that is a  finite set is called a finite patch.
Given a tiling $\calT$, a patch $\calP$ is said to be \emph{$\calT$-legal} if there is an
$x\in\R^d$ such that $\calP+x\subset\calT$, where
for a patch $\calP$ in $\R^{d}$ and an $x\in\R^{d}$, we set $\calP+x=\{T+x\mid T\in\calP\}$.
A vector $x\in\R^{d}$ such that there is a $T\in\calT$ with $T+x\in\calT$ is called a 
\emph{return vector} for $\calT$.\label{def_return_vector}
The set of all return vectors for $\calT$ is denoted by $\Xi(\calT)$.

For a patch $\calP$ and a  subset $S$ of $\R^d$, we define a new patch $\calP\sqcap S$ via
\begin{align}
      \calP\sqcap S=\{T\in\calP\mid\supp T\cap S\neq\emptyset\}.
      \label{def_sqcap_cutting-off}
\end{align}

Define another patch $\calP\sci S$ via
     \begin{align}
         \calP\sci S=\{T\in\calP\mid \supp T\subset S\}.
         \label{def_sci_cutting-off}
     \end{align}

Next, we define the densities for patches. We first define van Hove sequences.

\begin{defi}
      A sequence $(A_{n})_{n=1,2,\cdots}$ of measurable subsets of $\R^{d}$ with positive Lebesgue measures is called a 
      van Hove sequence if, for each compact $K\subset\R^{d}$,
      we have
      \begin{align*}
             \lim_{n}\frac{\vol(\partial^{K}A_{n})}{\vol(A_{n})}=0,
      \end{align*}
      where $\vol$ denotes the Lebesgue measure and
      \begin{align*}
             \partial^{K}A_{n}=((K+A_{n})\setminus A_{n}^{\circ})\cup((-K+\overline{\R^{d}\setminus A_{n}})\cap A_{n}).
      \end{align*}

      The \emph{density} $\dens_{(A_{n})}\calP$ of a patch $\calP$
     with respect to a van Hove sequence
      $(A_{n})_{n}$ is defined via
      \begin{align}
             \dens_{(A_{n})_{n}}\calP=
             \limsup_{n\rightarrow\infty}\frac{1}{\vol (A_{n})}\vol((\supp\calP)\cap A_{n}).
             \label{def_density}
      \end{align}

      We say that
      a patch $\calP$ has \emph{zero density} if, for any van Hove sequence $(A_{n})_{n}$, the 
      density $\dens_{(A_{n})_{n}}\calP$ is zero.
\end{defi}

       If the diameters of the tiles in $\calP$ are bounded from above by some uniform constant, 
       we have
            \begin{align*}
             \dens_{(A_{n})_{n}}\calP             &= \limsup_{n\rightarrow\infty}\frac{1}{\vol(A_{n})}\vol(\supp(\calP\sqcap A_{n}))\\
             &= \limsup_{n\rightarrow\infty}\frac{1}{\vol(A_{n})}\vol(\supp(\calP\sci A_{n})).
             \end{align*}

The corresponding dynamical system for a tiling is an important object.
To define it, we need to define  the \emph{local matching topology}
on the set of all patches, as follows:
for patches $\calP_1$ and $\calP_2$, we define a set $\Delta(\calP_1,\calP_2)$ to be the
set of all real numbers $\e$ between $0$ and $\frac{1}{\sqrt{2}}$ such that there are
$x_1,x_2\in B_{\e}$ with
\begin{align*}
       (\calP_1+x_1)\sqcap B_{1/\e}=(\calP_2+x_2)\sqcap B_{1/\e}.
\end{align*}
Define a metric $\rho$ on the space of all patches in $\R^d$ via
\begin{align*}
      \rho(\calP_1,\calP_2)=\inf\Delta(\calP_1,\calP_2)\cup\left\{\frac{1}{\sqrt{2}}\right\}.
\end{align*}
Here, the choice of the number $1/\sqrt{2}$ makes it easier to prove the triangle inequality.
The topology defined by $\rho$ is called the local matching topology.
If we instead use $\sci$ as a cutting-off operation, the resulting topology is the same on
a large space of patches.
It is known that this metric is complete (\cite{Nagai_local-matching-top}).

Given a tiling $\calT$, define the \emph{continuous hull} $X_{\calT}$ via
\begin{align*}
       X_{\calT}=\overline{\{\calT+x\mid x\in\R^d\}},
\end{align*}
where the closure is taken with respect to the local matching topology.
There is a \textcolor{red}{necessary and} sufficient condition, \emph{finite local complexity (FLC)}, 
 for the continuous hull to be compact.
A tiling $\calT$ has FLC if, for each compact $K\subset\R^{d}$, the set
\begin{align*}
        \{\calT\sqcap (K+x)\mid x\in\R^{d}\}
\end{align*}
is finite up to translation. (That is, if we identify two patches that are translates of each other, then
the set is finite). 
The
FLC of a tiling \textcolor{red}{is equivalent to}
 the compactness of its continuous hull (\cite[Proposition 5.4]{Baake-Grimm1}).
By taking $K=\{0\}$, we also see
the FLC implies that there are only finitely many tiles in $\calT$ up to
translation. 
 The group $\R^d$ acts continuously on $X_{\calT}$ via
\begin{align*}
        X_{\calT}\times\R^d\ni(\calS,x)\mapsto \calS+x\in X_{\calT}.
\end{align*}
The pair of $X_{\calT}$ and this action is called \emph{the tiling dynamical system}
associated with $\calT$.
The tiling $\calT$  is said to be \emph{repetitive} if, for any finite $\calT$-legal patch $\calP$,
the translates of $\calP$ appear in $\calT$ with a uniformly bounded gap.
If $\calT$ is repetitive, then the tiling dynamical system $(X_{\calT},\R^{d})$ is minimal
(\cite[Proposition 5.4]{Baake-Grimm1}), \textcolor{red}{that is, every orbit is dense}.

In this article,  we assume that all tilings which appear have FLC, and
 there is one and only one invariant probability measure $\mu$ for 
each tiling dynamical system. The self-affine tilings defined via a primitive FLC substitutions, which
we focus on in this article, 
satisfy this assumption.

If, for a
vector $a\in\R^d$, there is a non-zero vector
$f\in L^{2}(\mu)$ such that for each $x\in\R^{d}$, the two maps
\begin{align*}
       X_{\calT}\ni\calS\mapsto f(\calS-x),\\
        X_{\calT}\ni\calS\mapsto e^{2\pi i\langle x,a\rangle}f(\calS)
\end{align*}
coincide $\mu$-almost everywhere, then we call $a$ 
an \emph{eigenvalue} for the tiling dynamical system $(X_{\calT},\R^d)$.
Here, $\langle\cdot,\cdot\rangle$
is the standard Euclidean inner product. In this case
the function $f$ is called an \emph{eigenfunction}.
If the function $f$ can be chosen to be continuous, we call $a$ a \emph{topological eigenvalue} and
$f$ a \emph{continuous eigenfunction}.
We say that
$\calT$ has \emph{pure discrete dynamical spectrum} if there is a complete orthonormal basis for
the Hilbert space $L^{2}(\mu)$ consisting of eigenfunctions.
\label{def_pure_discrete_spec}


\begin{figure}[htbp]
\begin{center}
\includegraphics{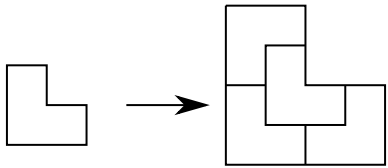}
\caption{the chair substitution}
\label{figure_chair_substi}
\label{default}
\end{center}
\end{figure}

One can construct interesting tilings via \emph{substitution rules}.
Intuitively, a substitution rule is a recipe for ``expanding a tile, followed by subdividing it so that
we obtain a patch.'' (See Figure \ref{figure_chair_substi}. This is an example of substitution
rule, called the chair substitution, that generates chair tilings in Figure \ref{Fig1}.)
A substitution rule by definition consists of
\begin{itemize}
 \item an \emph{alphabet} $\calA$, which is a finite set consisting of tiles in $\R^d$,
 \item an \emph{expansion map} $Q$, which is a linear transformation
        $Q\colon \R^d\rightarrow\R^d$ whose eigenvalues are all greater than $1$ in absolute values,
        and
 \item a map $\omega$ which sends each element $P$
       of $\calA$ to a patch $\omega(P)$ consisting
        of translates of elements of $\calA$ such that
        \begin{align*}
               \supp\omega(P)=Q(\supp P).
        \end{align*}
\end{itemize}
Often the map $\omega$ is also called a substitution rule.
The elements of $\calA$ are called \emph{proto-tiles}.
For the example of the chair substitution, the alphabet $\calA$ consists of the tile on the left-hand side
of Figure \ref{figure_chair_substi} and its rotations by degrees 90, 180, 270.
$Q=2I$ (where $I$ is the identity matrix) and the map $\omega$ sends each element of $\calA$
to the patch that is obtained in the way depicted in Figure \ref{figure_chair_substi}.

Given a substitution rule $\omega$ with an indexed alphabet $\calA=\{T_{1},T_{2},\ldots, T_{n}\}$,
we associate a substitution matrix $M_{\omega}$ whose $(i,j)$-element is the number of 
the occurrences
of the tile $T_{i}$ in $\omega(T_{j})$. The substitution $\omega$ is said to be \emph{primitive} if
the matrix $M_{\omega}$ is primitive, which means that
 its sufficiently large powers have only non-zero
elements.
The substitution $\omega$ is said to be \emph{irreducible} if the characteristic polynomial
of $M_{\omega}$ is
irreducible.

Given an $\omega$, we can define, in a natural way, 
 a patch $\omega(\calP)$ for a patch $\calP$ consisting of 
translates of tiles in $\calA$, by applying the same ``expanding and subdividing'' rule
 to each tile in 
$\calP$.
We can iterate $\omega$ and sometimes, as $n\rightarrow\infty$,
 $\omega^{n}(\calP)$ grows larger and larger and
converges to a tiling.
Intuitively, a tiling constructed in this way is called a self-affine tiling, of which
formal definition is given below.
 If a tiling $\calT$ consists of translates of tiles in $\calA$ and we have
$\omega(\calT)=\calT$, we call $\calT$ a fixed point of the substitution rule $\omega$.
If moreover
$\calT$ has FLC and is repetitive, we call $\calT$ a \emph{self-affine tiling}.
If the expansion map is of the form $\lambda I$ for some $\lambda>1$, where $I$ is the
identity map, then the resulting self-affine tiling is called a \emph{self-similar tiling}.
In this case, the number $\lambda$ is called the \emph{expansion factor}.
Some chair tilings are self-similar tilings with expansion factor $2$.

For a substitution rule $\omega$ of $\R^{d}$, we say $\omega$ has FLC if for each compact
$K\subset\R^{d}$, the set
\begin{align*}
        \{\omega^{n}(P)\sqcap (K+x)\mid n>0, P\in\mathcal{A}, x\in\R^{d}\}
\end{align*}
is finite up to translation. If $\omega$ has FLC, any repetitive fixed points for $\omega$
have FLC.

In general, given a tiling $\calT$ of $\R^{d}$ with FLC, 
we can obtain a discrete and closed subset $D$
of $\R^{d}$ by choosing one representative point from the support of each tile in $\calT$,
in such a way that the relative positions of representative points for translationally equivalent tiles
are always the same.
This $D$ is an example of Delone set, because (1)distances between any two points are 
uniformly bounded from below,
and (2)there is an $R>0$ such that any balls of $\R^{d}$ with radius $R$ contain points of $D$.
If $\calT$ is a self-affine tiling associated with
 a substitution rule $\omega$ with an expansion map $Q$, we can take a \emph{control point}\label{def_control-point} from each tile as a representative point
\cite{prag2}: we first choose a tile map $f\colon\calT\rightarrow\calT$ such that
$f(T)\in\omega(T)$ for each $T\in\calT$, and let $c(T)\in\supp T$ be the unique point in
$\bigcap_{n=1}^{\infty}Q^{-n}(f^n(T))$. The resulting Delone set 
$D=\{c(T)\mid T\in\calT\}$ is well-behaved in the sense that
(1) we have $Q(D)\subset D$
(because $Q(c(T))=c(f(T))$),
and (2) with an additional condition on $\calT$, the set $D$ is a \emph{Meyer set}\label{def_Meyer-set}, 
which means that for some
neighborhood $U$ of $0\in\R^{d}$, we have
\begin{align*}
      ((D-D)-(D-D))\cap U=\{0\}.
\end{align*}
The tilings $\calT$ with the latter property (2) for some choice of representative points 
are said to have the
\emph{Meyer property}.

A condition for a self-affine tiling $\calT$ to have the Meyer property is given in \cite{LeeSol:08}.
We say that a set $\Lambda$ of algebraic integers forms a \emph{Pisot family} if for any 
$\lambda\in\Lambda$ and any algebraic conjugate $\mu$ of $\lambda$ with $|\mu|\geq 1$,
we have $\mu\in\Lambda$.
\begin{thm}[\cite{LeeSol:08}]\label{thm_Lee-Solomyak_Meyer_property}
      Let $\calT$ be a self-affine tiling in $\R^{d}$ with an expansion map $Q$.
      Suppose $Q$ is diagonalizable over $\C$ and all the eigenvalues are algebraic conjugates and 
      have the same multiplicity. 
      Then $\calT$ has the Meyer property if and only if
      $\spec(Q)$ is a Pisot family, where $\spec(Q)$ is the set of all eigenvalues for $Q$.
\end{thm}

\subsection{Arithmetic progressions in tilings}

In this subsection we define arithmetic progressions in tilings and describe known results for them.

\begin{defi}
      Let $\calT$ be a tiling in $\R^{d}$, $\calP$ a non-empty finite patch,
      $n$ an integer greater than $0$
       and $x$ a vector in $\R^{d}\setminus\{0\}$.
      A set of tiles of the form
      \begin{align*}
             \bigcup_{k=1}^{n}\calP+kx
      \end{align*}
      is called a \emph{(one-dimensional) arithmetic progression of length $n$}.
      A set of tiles of the form
      \begin{align*}
           \bigcup_{k=1}^{\infty}\calP+kx
      \end{align*}
      or
      \begin{align*}
             \bigcup_{k=-\infty}^{\infty}\calP+kx
      \end{align*}
      is called a \emph{(one-dimensional) infinite arithmetic progression}.
      
      Let $\mathcal{B}=\{b_{1},b_{2},\ldots,b_{d}\}$ be a basis of $\R^{d}$,
       which is regarded as a vector space over $\R$.
      A set of tiles of the form
      \begin{align*}
             \bigcup_{(l_{1},l_{2},\cdots ,l_{d})\in\mathbb{Z}^{d}}\calP+\sum_{i=1}^{d}l_{i}b_{i}
      \end{align*}
      is called a \emph{full-rank infinite arithmetic progression}.
\end{defi}

In this paper, we discuss the existence/non-existence of various arithmetic progressions in
various tilings. We first note that the following two lemmas hold:

\begin{lem}\label{equivalence-for-one-dimAP}
       Let $\calT$ be a tiling in $\R^{d}$ that has FLC.
       Let $\calP$ be a non-empty, finite and $\calT$-legal patch and
       take an $x\in\R^{d}\setminus\{0\}$.
       Then the following conditions are equivalent:
       \begin{enumerate}
       \item for each $n>0$, an arithmetic progression $\bigcup_{k=1}^{n}\calP+kx$ is
                 $\calT$-legal.
        \item for some $\calS\in X_{\calT}$, an infinite arithmetic progression
                 $\bigcup_{k=1}^{\infty}\calP+kx$ is $\calS$-legal.
          \item for some $\calS\in X_{\calT}$, an infinite arithmetic progression
                   $\bigcup_{k=-\infty}^{\infty}\calP+kx$ is $\calS$-legal.
       \end{enumerate}
\end{lem}
\begin{proof}
     If the first condition is satisfied, we can take a sequence $x_1,x_2,\ldots\in\R^d$ such that
     for each $n$, the translate $\calT-x_n$ includes $\bigcup_{k=1}^n\calP+kx$.
     By FLC, \textcolor{red}{the hull
     $X_{\calT}$ is compact and}
     we can take a subsequence $(x_{n_j})_j$ such that
     $\calS=\lim_{j\rightarrow\infty}\calT-x_{n_j}$ is convergent. This $\calS$ includes the patch
     $\bigcup_{k=1}^{\infty}\calP+kx$.
     We can also take an $x_n$ such that $\calT-x_n$ includes
     $\bigcup_{k=-n}^n\calP+kx$. By the same argument, we can find an $\calS\in X_{\calT}$ that
     includes $\bigcup_{k=-\infty}^{\infty}\calP+kx$. We have proved that the first condition
     implies both the second and the third condition.
     
     Conversely, if the second or the third condition is satisfied, then for each $n$, the patch
     $\bigcup_{k=1}^n\calP+kx$ is $\calS$-legal. Since $\calS$ can be approximated by translates
     of $\calT$, we see $\bigcup_{k=1}^n\calP+kx$ is $\calT$-legal for each $n$.
\end{proof}

\begin{lem}\label{lem_equivalence_full-rankAP}
        Let $\calT$ be a tiling in $\R^{d}$ that has FLC and is repetitive.
        Let $\calP$ be a non-empty, finite and $\calT$-legal patch
        and $\mathcal{B}=\{b_{1},b_{2},\ldots b_{d}\}$ be a basis of $\R^{d}$.
        Then the following conditions are equivalent:
        \begin{enumerate}
        \item for each $n>0$, the set
                 \begin{align}
                        \bigcup_{(l_{1},l_{2},\ldots ,l_{d})\in\{1,2,\ldots, n\}^{d}}
                        \calP+\sum_{i=1}^{d}l_{i}b_{i}
                        \label{full-rank-finite-patch}
                 \end{align}
                  is a $\calT$-legal patch.
            \item The infinite arithmetic progression
                      \begin{align}
                        \bigcup_{(l_{1},l_{2},\ldots ,l_{d})\in\mathbb{Z}^{d}}
                        \calP+\sum_{i=1}^{d}l_{i}b_{i}
                              \label{full-rank-infinite-AP}
                      \end{align}
                      is $\calS$-legal for any $\calS\in X_{\calT}$.
        \end{enumerate}
\end{lem}

\begin{proof}
      If the first condition is satisfied, there is an $x_{n}\in\R^{d}$  for $n=1,2,\cdots$ such that
      $\calT-x_{n}$ contains a patch
                       \begin{align}
                        \bigcup_{(l_{1},l_{2},\ldots ,l_{d})\in\{-n,-n+1,\ldots, n\}^{d}}
                        \calP+\sum_{i=1}^{d}l_{i}b_{i}.
                        \label{full-rank-finite-patch}
                 \end{align}
      Since $X_{\calT}$ is a compact metric space, we can take  $n_{1}<n_{2}<n_{3}<\cdots$ such
      that $\calT-x_{n_{j}}$ converges as $j\rightarrow\infty$ to some $\calS\in X_{\calT}$.
      This $\calS$ contains \eqref{full-rank-infinite-AP}. Since the tiling dynamical system
         $(X_{\calT},\R^{d})$  is minimal, for each $\calR\in X_{\calT}$ there is a sequence
         $y_{1},y_{2},\cdots\in\R^{d}$ such that $\lim_{m\rightarrow\infty}\calS+y_{m}=\calR$.
         We can take a subsequence $(y_{m_{j}})_{j}$ that is convergent mod
         $\spa_{\Z}\mathcal{B}$. We still have $\calR=\lim_{j\rightarrow\infty}\calS+y_{m_{j}}$,
         and this $\calR$ contains a translate of the infinite arithmetic progression
         \eqref{full-rank-infinite-AP}.
\end{proof}

\begin{rem}
       Notice that (1)for a repetitive $\calT$,
       the existence/non-existence of a one-dimensional
       finite arithmetic progression
       $\bigcup_{k=1}^n\calP+kx$ does not depend on the choice of $\calS\in X_{\calT}$,
       because $X_{\calT}$ coincides with the local isomorphism class
       \cite[Proposition 5.4]{Baake-Grimm1},
       (2)the existence/non-existence of a
       one-dimensional infinite 
       arithmetic progression $\bigcup_{k=1}^{\infty}\calP+kx$
       may depend on the choice of the tiling from the hull
       (compare Lemma \ref{equivalence-for-one-dimAP}), and
       (3)the existence/non-existence of the full-rank infinite arithmetic progressions does not
       (Lemma \ref{lem_equivalence_full-rankAP}).
        In this paper, we discuss the existence/non-existence of finite arithmetic progressions (1)
         and of full-rank infinite arithmetic progressions (3).
         We discuss the existence and non-existence inside a repetitive fixed point $\calT$ 
         of a substitution which
         we initially consider, but the same existence/non-existence applies for any 
         $\calS\in X_{\calT}$.
\end{rem}

\begin{ex}\label{ex_chair_infinite-AP}
       Let $\omega$ be the chair substitution. 
              For $\calP_{0}$ and $x\neq 0$, as depicted in Figure \ref{figure_infinite-AP},
 there is an arithmetic progression of length
       $2^{k}$ at the $k$-th step of the substitution. This means that the chair tilings satisfy
       the equivalent conditions of Lemma \ref{equivalence-for-one-dimAP} for this $\calP_{0}$ and $x$.
       If we pick an arbitrary finite patch $\calP$ that appears in chair tilings, since a translate of
       $\calP$ is  included in $\omega^{l}(\calP_{0})$ for some $l>0$, we see for $\calP$ and
       $2^{l}x$, there are arbitrarily long (one-dimensional) arithmetic progressions in chair
       tilings. 
       
       \textcolor{red}{We can see that the chair tilings admit full-rank infinite arithmetic progressions
       by considering the modification of the chair tilings into tilings with decorated squares
      \cite[Section 4]{BMS-limit_periodic}.
      The supertiles of the decorated squares have common tiles except for the main diagonals.
      Since the supertiles are aligned in a translate of $2^n\mathbb{Z}^2$, 
      if we chose tiles from supertiles
      that are relatively the same position and are outside the
      main diagonals, we can find a full-rank infinite arithmetic progressions.
       }
\end{ex}

\begin{figure}[htbp]
\begin{center}
\includegraphics{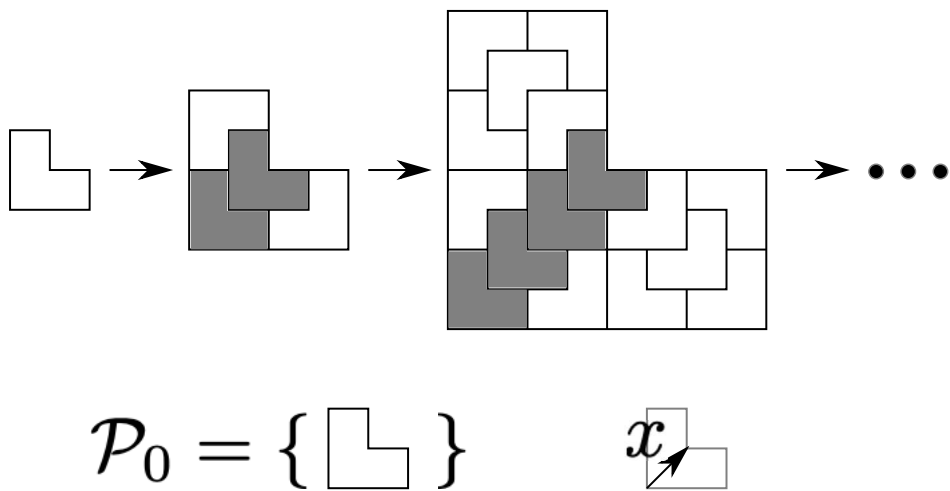}
\caption{Arbitrarily long arithmetic progressions in chair tilings}
\label{figure_infinite-AP}
\label{default}
\end{center}
\end{figure}

\begin{ex}\label{ex_table-substi}
       Let $\omega$ be the table substitution (\cite[Example 6.2]{Baake-Grimm1}). 
       As is seen at the center of Figure \ref{figure_table-tiling}, 
       there are arbitrarily long finite arithmetic progressions
       (in the horizontal direction)
       in the self-similar
       tilings for $\omega$. However, we will see by using
       Theorem \ref{full-rank-infiniteAP_implies_pp} that there
       are no full-rank infinite arithmetic progressions in those tilings, since they do not have
       pure discrete dynamical spectrum (\cite[Example 7.3]{Solomyak_Tiling}).
\end{ex}

\begin{figure}[htbp]
\begin{center}
\includegraphics{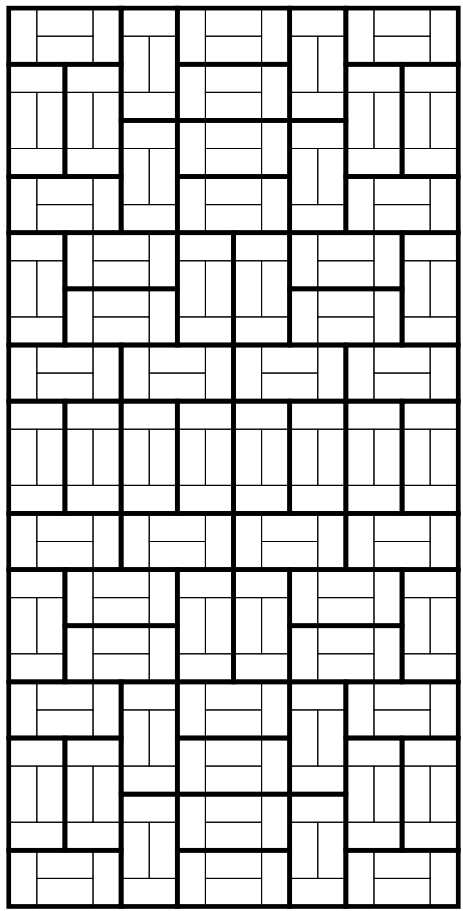}
\caption{A patch generated by the table substitution}
\label{figure_table-tiling}
\label{default}
\end{center}
\end{figure}

To prove the existence of arithmetic progressions in tilings, the following theorem is useful:
\begin{thm}[\cite{Fur},Theorem 2.6]
      Let $X$ be a compact metric space and $T_{1},T_{2},\ldots, T_{l}$ commuting continuous maps
      of $X$ to itself. Then there is a point $x\in X$ and a sequence $n_{k}\rightarrow\infty$ such that
      $T^{n_{k}}_{i}x\rightarrow x$ simultaneously for $i=1,2,\ldots, l$. 
\end{thm}

Using this multiple Birkhoff recurrence theorem,  the authors of \cite{Llave-Windsor} proved the
following theorem:
\begin{thm}[\cite{Llave-Windsor}, Theorem 2]\label{thm_Llave-Windsor}
    Let $\calT$ be a tiling in $\R^{d}$ that has FLC.
       Given an $\e>0$ and a finite set $F\subset\R^{d}$, there exist an
       $m\in\Zpo$ and a patch $\calP$ 
       such that
       \begin{enumerate}
       \item the support of $\calP$  covers a ball of radius $1/\e$, and
      \item for each $u\in F$ there is a vector $c\in\R^{d}$ such that $\|c\|<\e$ and
                \begin{align*}
                        \calP+mu+c\subset\calT.
                \end{align*}
       \end{enumerate}
\end{thm}

Given a tiling $\calT$ in $\R^{d}$ of FLC, an $x\in\R^{d}\setminus\{0\}$ and an $n>0$, by applying
Theorem \ref{thm_Llave-Windsor} for $F=\{0, x, 2x, \ldots, (n-1)x\}$ and an $\e>0$, we see
that
there is a patch $\calP$ that covers a ball of radius $1/\e$ and an $m\in\mathbb{Z}_{>0}$ such that
an ``almost arithmetic progression''
\begin{align*}
       \bigcup_{k=1}^{n}\calP+kmx+c_{k}
\end{align*}
is $\calT$-legal. (``Almost'' means that there is an error term $c_{k}$ with $\|c_{k}\|<\e$.)

If we start with a non-empty finite patch $\calQ$ and assume that $\calT$ is repetitive, then
if $\e$ is small enough every patch of $\calT$ that covers a ball of radius $1/\e$ contains a translate
of $\calQ$. If $\e$ is small enough in this sense, then $\calP$ contains a translate $\calQ+y$ of
$\calQ$, and so
\begin{align*}
      \bigcup_{k=1}^{n}\calQ+kmx+c_{k}
\end{align*}
is $\calT$-legal.

If $\calT$ has the 
Meyer property, we will have an ``exact'' arithmetic progression without the error terms.
Assuming that the
 tiling $\calT$ has the Meyer property (page \pageref{def_Meyer-set}),
then there is an $r>0$ such that
\begin{align*}
     (\Xi-\Xi)\cap B_{r}=\{0\},
\end{align*}
where $\Xi=\Xi(\calT)$ is the set of all return vectors (page \pageref{def_return_vector}).
If there is such an $r$, by taking $\e<r/4$ we see
\begin{align*}
           (k+1)mx+c_{k+1}-(kmx+c_{k})=kmx+c_{k}-((k-1)mx+c_{k-1})
\end{align*}
for each $k$, and so by denoting this vector by $z$ we see the patch
\begin{align*}
         \bigcup_{k=0}^{n-1}\calQ+kz
\end{align*}
is $\calT$-legal. We have proved the following corollaries:
\begin{cor}\label{cor_existence_one-dimAP}
      If a tiling $\calT$ is repetitive and has FLC and the Meyer property, then for any finite $\calT$-legal
      patch $\calP$ and $n>0$, there is a $z\in\R^{d}\setminus\{0\}$ such that
      \begin{align}
              \bigcup_{k=1}^{n}\calP+kz\label{AP_for_given_n}
      \end{align}
      is $\calT$-legal.
      The direction $z/\|z\|$ can be chosen arbitrarily close to any given vector of length $1$.
\end{cor}

\begin{cor}\label{cor_delaLlave}
        Let $\calT$ be a self-affine
         tiling as in Theorem \ref{thm_Lee-Solomyak_Meyer_property} and assume
        $\spec(Q)$ is a Pisot family. Then for any finite $\calT$-legal patch $\calP$ and $n>0$, there
        is a $z\in\R^{d}\setminus\{0\}$ such that the arithmetic progression \eqref{AP_for_given_n}
        is $\calT$-legal. The direction $z/\|z\|$ can be chosen to be arbitrarily close to any given
        vector of length 1.
\end{cor}

Recently, this corollary was generalized in \cite{KST}, as follows:
\begin{thm}[\cite{KST}]
      Let $\Lambda$ be a Meyer set in $\R^{d}$. Then, for all positive integers $r$  and $k$, 
      there is an $R>0$ such that, for any coloring of points in $\Lambda$ in $r$ colors and 
      any $x\in\R^{d}$, there is a monochromatic arithmetic progression of length $k$ inside
      $\Lambda\cap B(x,R)$.
\end{thm}
This van der Waerden-type theorem is more general than Corollary \ref{cor_delaLlave}
because it is valid for arbitrary Meyer sets and colorings.
Corollary \ref{cor_delaLlave} is obtained from this theorem, except for the last statement
on $z/\|z\|$ in
the corollary, when we apply this theorem to the Meyer set $\{x\in\R^{d}\mid \calP+x\subset\calT\}$.

Corollary \ref{cor_existence_one-dimAP}
proves the existence of arithmetic progressions, \emph{given a $\calP$ and the length
$n$ of the
 arithmetic progression}. Comparing this to Lemma \ref{equivalence-for-one-dimAP}, we then ask:

\begin{prob}\label{question_one-dim-AP}
\textcolor{red}{
Take a tiling $\calT$, a patch $\calP$ and a non-zero return vector $x$.
Decide if the arithmetic progression
     \begin{align}
          \bigcup_{k=1}^{n}\calP+kx\label{AP_given_x}
     \end{align}
      is always 
     $\calT$-legal for every $n>0$.}
\end{prob} 

\textcolor{red}{
For a self-affine tiling, we can check all the patterns which can be observed in ball-shaped
windows of a certain radius $R>0$. If for some $\calP$ and $x$, the arithmetic progressions
always stop in the patterns, then we see the arithmetic progressions stop at a fixed finite step
anywhere in the tiling. However, we can check this only for a certain choice of $R$ and there are
infinite possibilities of $R$. In some cases, we can use induction to show there are arbitrarily long
arithmetic progressions (Example \ref{ex_chair_infinite-AP} and Example \ref{ex_table-substi}) for
some $\calP$ and $x$.
How about the Penrose tilings? We will show that for the Penrose tilings and many self-affine tilings,
any arithmetic progression always stops at a fixed finite step, in Section \ref{section_non-existence}.}

By the same argument as Corollary \ref{cor_delaLlave}, we can prove the following ``full-rank version'':

\begin{thm}
        Let $\calT$ be a tiling as in Theorem \ref{thm_Lee-Solomyak_Meyer_property} and assume
        $\spec(Q)$ forms
         a Pisot family. Then for any finite $\calT$-legal patch $\calP$ and an $n>0$, there
        is a basis $\{b_{1},b_{2},\ldots ,b_{d}\}$ of $\R^{d}$ such that the patch
        \begin{align*}
               \bigcup_{(l_{1},l_{2},\ldots, l_{d})\in\{1,2,\ldots,n\}^{d}}\calP+\sum_{j=1}^{d}l_{j}b_{j}
        \end{align*}
        is $\calT$-legal.
\end{thm}

       The set $\{b_{1}/\|b_{1}\|,b_{2}/\|b_{2}\|,\ldots ,b_{d}/\|b_{d}\|\}$ can be chosen arbitrarily close
        to any subset of cardinality $d$ of the unit sphere that form a basis of $\R^{d}$, because
        we can take at least one $\{b_1,b_2,\ldots ,b_d\}$ which forms a basis.

Then the natural question is the following (compare Lemma \ref{lem_equivalence_full-rankAP}) :
\begin{ques}\label{question_full-rank-AP}
      \textcolor{red}{
      Which tilings $\calT$ admit full-rank infinite arithmetic progressions, and
      which $\calT$ do not?
      }
\end{ques}

We will address this question for self-affine tilings in Section \ref{section_full-rank-AP}.
Under a certain assumption, the existence of full-rank infinite arithmetic progressions is equivalent to
the pure discrete spectrum of the associated tiling dynamical system, in which case
we can actually prove that the tiling is limit-periodic.

\section{The non-existence of arbitrarily long arithmetic progressions}
\label{section_non-existence}

In this section, we prove several results on the non-existence of arbitrarily long arithmetic progressions
that have a fixed distance,
in a self-affine tiling, and so give a partial answer to Problem \ref{question_one-dim-AP}.

\subsection{An argument using continuous eigenfunctions}

In this subsection we prove Theorem \ref{thm_using_conti_eigenft}, which states that
there are not arbitrarily long arithmetic progressions for a fixed difference $v$, 
if the expansion factor is
irrational and Pisot. 

\begin{thm}\label{thm_using_conti_eigenft}
       Let $\omega$ be a primitive and FLC tiling substitution in $\R^d$ with an expansion factor
     $\lambda$ that is irrational and Pisot. Let $\mathcal{T}$ be a fixed point of
     $\omega$. Assume $\mathcal{T}$ is repetitive. Then for any
      $\mathcal{T}$-legal, nonempty and finite patch $\mathcal{P}$ and
      $v\in\mathbb{R}^d\setminus\{0\}$, there is $n_0\in\mathbb{N}$ such that
      $\bigcup_{n=0}^{n_0}(\mathcal{P}+nv)$ is not $\mathcal{T}$-legal.
\end{thm}

\begin{rem}
      If we drop the assumption that $\lambda$ is irrational, there is a counterexample for the 
      statement, as we have shown in Example \ref{ex_chair_infinite-AP}.
      We do not know if there is a counterexample if we drop the assumption that $\lambda$
      is Pisot. For the one-dimensional case, there is no arbitrarily long arithmetic progressions if
      $\lambda$ is irrational (Theorem \ref{thm_complete_picture_for_self-similar}).
      \textcolor{red}{Note that $\calT$ is non-periodic because the expansion factor is irrational so that
      the frequency for a tile is irrational.}
\end{rem}

We divide the proof into several parts, as follows.
First we define $A_{\calP}$, which will play an important role in the proof.

\begin{defi}
       Let $\calT$ be a tiling and $\calP$ be a patch. We consider a set
       \begin{align*}
               A_{\calP}=\{\calS\in X_{\calT}\mid \calP\subset\calS\}.
       \end{align*}
       Although this is dependent on $X_{\calT}$, we denote it just by $A_{\calP}$, since in this subsection
       $\calT$ is fixed.
\end{defi}

For each $\calS\in X_{\calT}$, we have $\calS\supset\calP$ if and only if $\calS\in A_{\calP}$.

Note that we have defined the notation $\sci$ in \eqref{def_sci_cutting-off} in page \pageref{def_sci_cutting-off}.
In the next lemma, the important point is that we can take $R$ that is independent of $x$.

\begin{lem}\label{lem_existence_R}
      Let $\mathcal{T}$ be a tiling of $\mathbb{R}^d$
      which has FLC.
      Let $f\colon X_{\mathcal{T}}\rightarrow\mathbb{T}$ be a continuous eigenfunction.
      Then for any $\varepsilon>0$ there is an $R>0$ such that the conditions
      \begin{itemize}
       \item $\mathcal{S}_1,\mathcal{S}_2\in X_{\mathcal{T}}$,
	 \item $x\in\mathbb{R}^d$, and
	  \item $\mathcal{S}_1\sci B(x,R)=\mathcal{S}_2\sci B(x,R)$
      \end{itemize}
       imply $|f(\mathcal{S}_1)-f(\mathcal{S}_2)|<\varepsilon$.
       Therefore, given a continuous eigenfunction $f$ and an $\e>0$,  if $R>0$ is large enough and
       a patch $\calP$ covers a ball of radius $R$, then
      the diameter of $f(A_{\calP})$ is less than $\e$.
\end{lem}
\begin{proof}
      Since $f$ is continuous,
      there is an $R>0$ such that if $\mathcal{S}_1,\mathcal{S}_2\in X_{\mathcal{T}}$ and
      \begin{align*}
           \mathcal{S}_1\sci B(0,R)=\mathcal{S}_2\sci B(0,R),
      \end{align*}
      then $|f(\mathcal{S}_1)-f(\mathcal{S}_2)|<\varepsilon$.
      In order to prove the statement of this lemma, take
      $\mathcal{S}_1,\mathcal{S}_2\in X_{\mathcal{T}}$ and an $x\in\mathbb{R}^d$, and
     assume
      \begin{align*}
           \mathcal{S}_1\sci B(x,R)=\mathcal{S}_2\sci B(x,R).
      \end{align*}
       Then
      \begin{align*}
            (\mathcal{S}_1-x)\sci B(0,R)=(\mathcal{S}_2-x)\sci B(0,R),
      \end{align*}
      and so $|f(\mathcal{S}_1-x)-f(\mathcal{S}_2-x)|<\varepsilon$.
       Since we have
       \begin{align*}
       |f(\mathcal{S}_1)-f(\mathcal{S}_2)|=
     |e^{2\pi i\langle a,x\rangle}f(\mathcal{S}_1)-
	 e^{2\pi i\langle a,x\rangle}f(\mathcal{S}_2)|	=
	|f(\mathcal{S}_1-x)-f(\mathcal{S}_2-x)|,
       \end{align*}
      we have proved the first claim.

       If $\calP$ is a patch that covers $B(x,R)$ for some $x\in\R^{d}$ and
       $\calS_{1},\calS_{2}\in A_{\calP}$, then $\calS_{1}\sci B(x,R)=\calS_{2}\sci B(x,R)$, and so
       $|f(\calS_{1})-f(\calS_{2})|<\e$. This means that
       the diameter of $f(A_{\calP})$ is less than $\e$.
 \end{proof}

\begin{proof}[The proof for Theorem \ref{thm_using_conti_eigenft}]
     Let $\{b_1,b_2,\ldots b_d\}$ be a basis of $\mathbb{R}^d$ consisting of
     topological eigenvalues for $(X_{\mathcal{T}},\mathbb{R}^d)$.
     Such a basis exists by \cite[the last part of page 13]{Solomyak_eigenfunction}.
     
     For a $\calP$ and a $v$ as in the statement of the theorem, there is an $i$ such that
     $\langle b_i,v\rangle\neq 0$. Let $f$ be a $\mathbb{T}$-valued
     continuous eigenfunction for $b_{i}$.
     
     To prove the statement, let us assume on the contrary that for each $n>0$, the patch
     $\bigcup_{k=0}^{n}\calP+kv$ is $\calT$-legal. We will show this leads to a contradiction.
     Since $\omega(\calT)=\calT$, for each $n>0$ and $m>0$, the patch
     $\calP_{n,m}=\bigcup_{k=0}^{n}\omega^{m}(\calP)+\lambda^{m}kv$ 
     is $\calT$-legal, which means that
     there is an $\calS_{n,m}\in X_{\calT}$ such that $\calS_{n,m}\supset\calP_{n,m}$.
     By the definition of $A_{\omega^{m}(\calP)}$, for each $n,m>0$ and $k=0,1,\ldots ,n$, we have
     \begin{align*}
             \calS_{n,m}-\lambda^{m}kv\in A_{\omega^{m}(\calP)}.
     \end{align*}
     This implies that
     \begin{align*}
              f(\calS_{n,m}-\lambda^{m}kv)=e^{2\pi i\langle b_{i},\lambda^{m}kv\rangle}f(\calS_{n,m})\in
              f(A_{\omega^{m}(\calP)}).
     \end{align*}
     If $m$ is large enough, then
      the number $\langle b_{i},\lambda^{m}v\rangle$ is irrational, and so the diameter of 
      $f(A_{\omega^{m}(\calP)})$ is larger than $1/2$. 
      On the other hand, by Lemma \ref{lem_existence_R}, if $m$ is large enough, then the
      diameter of $f(A_{\omega^{m}(\calP)})$ can be arbitrarily small. We have a contradiction and have proved
      the theorem.
%
\end{proof}

\begin{ex}
      The Penrose tilings are essentially the same (MLD) as the Robinson triangle tilings, which are 
      generated by a substitution with an expansion factor $\frac{1+\sqrt{5}}{2}$.
      This is an irrational Pisot number and, by Theorem \ref{thm_using_conti_eigenft}, we see that
      there are no arbitrary long arithmetic progressions when we fix $\calP$ and $v$, 
      \textcolor{red}{both in the Penrose tilings and the Robinson triangle tilings.}
\end{ex}

\subsection{An argument using an internal space}

In this subsection, we introduce another method to show the non-existence of  arbitrarily long
 arithmetic progressions 
in a given self-affine tiling (Theorem \ref{thm_internal_space}).
The argument relies on \cite[Theorem 4.1]{PisotFam} and is valid for a more general class of
self-affine tilings, whereas Theorem \ref{thm_using_conti_eigenft} is valid only for self-similar
tilings.

Let $\Tk$ be a self-affine tiling with an alphabet $\calA=\{T_1,T_{2},\ldots T_{m}\}$,
an expansion map $Q$ and a primitive FLC substitution map $\omega$.
We assume $Q$ is diagonalizable
over $\mathbb{C}$ and
all eigenvalues are algebraic conjugates of the same multiplicity $J$.
We select control points (page \pageref{def_control-point})
to obtain a Delone set
$\Lam=\bigcup_{i=1}^m \Lam_i$, where each $\Lambda_i$ is the set of all
control points of the tiles of type $i$.

Since $Q$ is diagonalizable and the multiplicities of the eigenvalues of $Q$ are always $J$,
if the real eigenvalues are $\lambda_1,\lambda_2,\ldots,\lambda_s$ and
imaginary eigenvalues are $\mu_1=a_1+b_1\sqrt{-1},\overline{\mu_1}=a_1-b_1\sqrt{-1},\ldots,
\mu_t=a_t+b_t\sqrt{-1},\overline{\mu_t}=a_t-b_t\sqrt{-1}$,
there is a basis $\mathcal{B}$ of $\R^d$ such that
the matrix $[Q]_{\mathcal{B}}$ of $Q$ with respect to $\mathcal{B}$ is
\begin{align*}
       [Q]_{\mathcal{B}}=
       \begin{bmatrix}
	     A & 0 & 0&\cdots &0 \\
	     0 & A & 0 & \cdots & 0\\
	     \vdots & & \ddots & &\vdots \\
	     0 & \cdots &   & & A
       \end{bmatrix},
\end{align*}
where
\begin{align*}
      A=
      \begin{bmatrix}
            \lambda_1& 0         & 0      &\cdots    &    &   &      &   &0\\
            0        & \lambda_2 & 0      &\cdots    &    &   &      &   &0\\
            \vdots   &           & \ddots &          &    &   &      &   &0\\
                     &           &        &\lambda_s &    &   &      &   &\\
                     &           &        &          &a_1 &-b_1&\cdots&  &\\
                     &           &        &          &b_1 &a_1&      &   &\\
                     &           &        &          &    &   &\ddots&   &\\
            0        &           &        &          &    &   &      &a_t&-b_t\\
            0        &           &        &          &    &   &      &b_t&a_t
      \end{bmatrix}.
\end{align*}
(This is the real canonical form of $Q$.)
Let $\alpha_j$ ($j=1,2,\ldots, J$) be the vector of $\R^d$ such that the representation
with respect to the basis $\mathcal{B}$ is
\begin{align*}
       (0,0,\ldots,0,\underbrace{1,1,\ldots 1,}_{\text{from $(j-1)m+1$ to $jm$}} 0,\ldots 0)
\end{align*}
where $m=s+2t$. 

By \cite[Theorem 4.1]{PisotFam}, there exists an isomorphism $\rho:\R^d\to \R^d$
such that $\rho Q=Q \rho$ and
\begin{align*}
\Lam \subset \rho( \Z[Q] \alpha_1 + \dots +\Z[Q] \alpha_J) .
\end{align*}
This is the first tool that we use in Theorem \ref{thm_internal_space}.

The other tool is the following.
By \cite[Lemma 6.5]{Solomyak_Tiling}, there 
exists a finite set $F\subset \Lambda-\Lambda$ such that
for any element $x \in\Lambda-\Lambda$, there are an $n>0$ and $w_1,w_2,\ldots,w_n\in F$
with
\begin{align} \label{translation-vector-form}
 x=\sum_{i=0}^{n} Q^{i} w_i.
\end{align}

In fact, $x=c(T)-c(S)$ for some tiles $T$ and $S$ in $\calT$. By repetitivity, we may
assume that $T$ and $S$ are both in $\omega^n(U)$ for some $n>0$ and $U\in\calT$.
There are $T^{(0)}=U, T^{(1)}\in\omega(T^{(0)}), T^{(2)}\in\omega(T^{(1)}),\ldots$
such that
$T=T^{(n)}\in\omega(T^{(n-1)})$. We have
\begin{align*}
        c(T)=\sum_{k=0}^{n-1}Q^{n-1-k}(c(T^{(k+1)}))-Q^{n-k}(c(T^{(k)}))+Q^n(c(T^{(0)})).
\end{align*}
There are also $S^{(0)}=U, S^{(1)}\in\omega(S^{(0)}), S^{(2)}\in\omega(S^{(1)}),\ldots$
such that
$S=S^{(n)}\in\omega(S^{(n-1)})$. We have
\begin{align*}
        c(S)=\sum_{k=0}^{n-1}Q^{n-1-k}(c(S^{(k+1)}))-Q^{n-k}(c(S^{(k)}))+Q^n(c(S^{(0)})).
\end{align*}
Therefore,
\begin{align*}
      c(T)-c(S)=\sum_{k=0}^{n-1}Q^{n-1-k}(c(T^{(k+1)})-Q(c(T^{(k)})))
      -\sum_{k=0}^{n-1}Q^{n-1-k}(c(S^{(k+1)})-Q(c(S^{(k)}))).
\end{align*}
Each of $c(T^{(k+1)})-Q(c(T^{(k)}))$ and $c(S^{(k+1)})-Q(c(S^{(k)}))$ is
a difference of tiles inside a level-one supertile.

With these machinery, we can prove the following theorem.
The idea is to use a kind of internal space.
For a model set $D$ \cite[Chapter 7]{Baake-Grimm1} with a Euclidean internal space,
if there are arbitrarily long arithmetic progressions $x, x+y, x+2y,\ldots x+ny\in D$ with 
\textcolor{red}{a
fixed} $y\neq 0$, 
then since $y^*\neq 0$, for a large $n$,
the corresponding point $x^*+ny^*$ should be outside the window and
we have a contradiction.
Here, we construct a similar Euclidean ``internal space'' and have a similar argument to obtain
a contradiction if there are arbitrarily long arithmetic progressions.

\begin{thm}\label{thm_internal_space}
Let $\Tk$ be a self-affine tiling of $\R^d$ with an expansion map $Q$ and a primitive FLC
substitution rule $\omega$.
Suppose $Q$ is diagonalizable over $\mathbb{C}$ and all the eigenvalues are algebraic
 conjugates with the same multiplicity $J$.
 If there exists an algebraic conjugate $\beta$
of eigenvalues of $Q$ such that $|\beta|<1$, 
then for any return vector
$x\in\mathbb{R}\setminus\{0\}$, there is an $n_0\in\mathbb{N}$ such that
$T+nx\notin\Tk$ for all $n>n_0$ and $T\in\Tk$.
 Therefore, for each $T\in\calT$ and $x\neq 0$, there is an $n_0>0$ such that
   $\bigcup_{k=1}^{n_0}T+kx$ is not $\Tk$-legal.
\end{thm}

\begin{proof}
Without loss of generality, we may assume that $\rho$ is the identity.

For each $q_1(x),q_2(x),\cdots q_J(x)\in\mathbb{Z}[x]$, the representation of
 the vector $q_1(Q)\alpha_1+\cdots +q_J(Q)\alpha_J$ with respect to the basis
 $\mathcal{B}$ is
 \begin{align*}
      \begin{bmatrix}
            q_1(A)\boldsymbol 1\\
            q_2(A)\boldsymbol 1\\
            \vdots   \\
            q_J(A)\boldsymbol 1
      \end{bmatrix}
 \end{align*}
 where $\boldsymbol 1\in\R^m$ is the vector whose entries are all 1.
 Define a well-defined injective map $\Phi\colon \mathbb{Z}[Q]\alpha_1+\cdots+\mathbb{Z}[Q]\alpha_J\rightarrow\mathbb{C}^J$ by
 \begin{align*}
      \begin{bmatrix}
            q_1(A)\boldsymbol 1\\
            q_2(A)\boldsymbol 1\\
            \vdots   \\
            q_J(A)\boldsymbol 1
      \end{bmatrix}
      \mapsto
      \begin{bmatrix}
             q_1(\beta)\\
             q_2(\beta)\\
             \vdots\\
             q_J(\beta)
      \end{bmatrix}.
 \end{align*}

 Since the set $F$ in (\ref{translation-vector-form}) is finite and $|\beta|<1$, we
can observe that the image of $\Lambda$ by $\Phi$ is bounded. If $x\in\Lambda-\Lambda$ and $x\neq 0$,
 then since $\Phi(x)\neq 0$, there is an $n_0$ such that $\|\Phi(n_0x)\|=\|n_0\Phi(x)\|$ is
 greater than the diameter of $\Phi(\Lambda)$. This implies that if  $y\in\Lambda$, then
 for each $n>n_0$, we have
$\Phi(y+nx)\notin\Phi(\Lambda)$ and
 $y+nx\notin\Lambda$. 
\end{proof}

\section{The existence and non-existence of full-rank infinite arithmetic progressions}
\label{section_full-rank-AP}

In this section, we investigate Question \ref{question_full-rank-AP} for self-affine
tilings in $\R^{d}$.
In particular, we prove that, under additional assumptions, the existence of full-rank
infinite arithmetic progressions implies the pure discrete spectrum of the tiling
(Theorem \ref{full-rank-infiniteAP_implies_pp}), and the converse
(Corollary \ref{cor_case_Xi_generates_lattice},
 Corollary \ref{cor_general_sufficient_condition_existence_full-rankAP},
 Theorem \ref{thm_self-similar_inf-fullrank-AP}).
Then the conclusion is strengthened by
showing that the tiling is in fact limit periodic (Theorem \ref{thm_limit_periodic},
Corollary \ref{cor_sufficient_conditions_for_limit-period}).

 \subsection{Full-rank  infinite arithmetic progressions imply pure discrete spectrum}
The following setting is assumed throughout the subsection.
\begin{setting}\label{setting_infap_implies_pp}
     In this subsection, $\calT$ is a self-affine tiling with an expansion map $Q$ and
     a primitive FLC substitution map $\omega$.
\end{setting}

The key notion here is the overlap coincidence, which is under an mild assumption equivalent to
pure discrete spectrum.

\begin{defi}[\cite{Solomyak_Tiling}] \label{def-overlap}  A triple $(T,y,S)$,
with $T,S \in \calT$ and $y \in \Xi(\calT)$, is called an {\em overlap}
if the intersection
$\supp(y+T)\cap \supp(S)$ has a non-empty interior.
We say that two overlaps $(T,y,S)$ and $(T',y',S')$ are {\em equivalent} if
for some $g\in \R^d$ we have $y+T = g+y'+T',\ S = g+S'$. Denote by
$[(T,y,S)]$ the equivalence class of an overlap. 
An overlap $(T,y,S)$ is a
{\em coincidence} if $y+T = S$. The {\em support} of an overlap
$(T,y,S)$ is $\supp(T,y,S) = \supp(y+T)\cap \supp(S)$.
\end{defi}

We define the {\em subdivision graph $\Gk_{\Ok}(\calT)$ for overlaps}.
Its vertices are the equivalence classes of overlaps for $\calT$. Let $\Ok=(T,y,S)$ be
an overlap.  We will specify directed edges
leading from the equivalence class $[\Ok]$. 
$\om(y+T) = Q y+\om(T)$ is a patch of $Q y+\calT$, and $\om(S)$ is a $\calT$-patch,
and moreover,
$$
\supp(Q y+\om(T)) \cap \supp(\om(S)) = Q (\supp(T,y,S)).
$$
For each pair of tiles $T'\in \om(T)$ and $S'\in \om(S)$ such that
$\Ok':= (T',Q y,S')$ is an overlap, we draw a directed edge from $[\Ok]$ to
$[\Ok']$.

The following equivalence is useful when we discuss overlaps.

\begin{lem}
      For an overlap $\mathcal{O}=(T,y,S)$, the following conditions are equivalent:
      \begin{enumerate}
      \item there is an $n>0$ such that $\omega^{n}(T+y)\cap\omega^{n}(S)\neq\emptyset$;
        \item from the equivalence class $[\mathcal{O}]$ there is a path that leads to the equivalence
        class of a coincidence.
      \end{enumerate}
\end{lem}

We also use the following two lemmas concerning overlaps.

\begin{lem}[\cite{LMS}, Lemma A.8]    \label{equivclass_overlap_finite}
    Assume that $\Xi(\calT)$ is a Meyer set.
     Then the number of equivalence classes of overlaps for $\calT$ is finite.
\end{lem}

We set
\begin{align}
      D_{x}=\calT\cap (\calT+x)
      \label{def_Dx}
\end{align}
for a vector $x\in\R^{d}$.

The following lemma is slightly different from \cite[Lemma A.9]{LMS}, but the proof is the same.
\begin{lem} [\cite{LMS}, Lemma A.9]\label{equivalence-to-coincidence}
Assume that $\Xi(\calT)$ is a Meyer set.
Let $x\in \Xi(\calT)$. The following are equivalent:

 {\em (i)} $\lim_{n\to\infty} \dens(D_{Q^n x})=1$;

 {\em (ii)} $1-\dens(D_{Q^n x}) \le Cr^n$, $n\ge 1$,
for some $C > 0$ and $r\in (0,1)$;

 {\em (iii)} From each vertex $\Ok$ of the graph $\Gk_{\Ok}(\calT)$
such that $\Ok=[(T,Q^kx,S)]$ for some $T,S\in\calT$ and $k>0$,
 there is a path
leading to a coincidence.
\end{lem}

\begin{thm} [\cite{Solomyak_Tiling},Theorem 6.1] \label{sufficient-condition-ppd}
Assume that
$\Xi(\calT)$ is a Meyer set.
If there exists a basis $\mathcal{B}$ for $\R^d$ such that for all $x \in \mathcal{B}$,
\[  \sum_{n=0}^{\infty} (1 - \dens(D_{Q^n x})) < \infty \,, \]
then the tiling dynamical system $(X_{\calT}, \mu, \R^d)$ has pure discrete spectrum.
\end{thm}

By these two results we have the following corollary:
\begin{cor}\label{cor_sufficient_condition_for_purepoint}
      Assume that $\Xi(\calT)$ is a Meyer set.
      Let $\mathcal{B}$ be a basis of $\R^{d}$ that is included in $\Xi(\calT)$.
      Suppose that for each overlap $(S,Q^{k}y,T)$ with $y\in\mathcal{B}$ and $k>0$, 
      there is a path from $[(S,Q^{k}y,T)]$ to
      a coincidence (that is, there is an $n>0$ such that
       $\omega^{n}(S+Q^{k}y)\cap\omega^{n}(T)\neq\emptyset$).
       Then the dynamical system $(X_{\calT},\mu,\R^{d})$ is pure discrete.
\end{cor}

%

\begin{thm}\label{full-rank-infiniteAP_implies_pp}
Suppose that there exist a tile $T_0 \in \calT$ and a lattice $L$ on $\R^d$
 such that $\{T_0 + v \ | \ v \in L\}  \subset \calT$ and $Q(L)\subset L$.
Then $(X_{\calT}, \mu, \R^d)$ has pure discrete spectrum.
\end{thm}

\begin{proof}
Note that $\Xi(\calT)$ is a Meyer set by the existence of a full-rank infinite arithmetic progression and
FLC: any return vector is written as the sum of the displacement to the nearest tile in a full-rank
infinite arithmetic progression, a vector inside $L$, and the displacement of a tile in
a full-rank infinite arithmetic progression to the target.
We also note that $L \subset \Xi(\calT)$.
We consider the overlaps which can occur from the translation vectors in $L$. 
For each such overlap $\mathcal{O}=(T, y, S)$, $y \in L$, we can find some integer $n \in \N$ such that 
the $n$-th inflation of the support of the overlap $\mathcal{O}$ contains at least one element of $\{T_0 + v \ | \ v \in L\}$, i.e. 
there exists $u \in L$ such that 
\[ \supp(T_0 + u) \subset Q^n (\supp(T, y, S)) . \] 
By the assumption that $\{T_0 + v \ | \ v \in L\}  \subset \calT$, 
\[ T_0 + u \in \omega^n(S).  \]
Since $Q L \subset L$, $T_0 + u - Q^n y \in \omega^n(T)$. 
Therefore, the overlap $\mathcal{O}= (T, y, S)$ admits a coincidence in the $n$-th iteration.
We can choose a set of the vectors $\mathcal{B}:=\{ y_1, \dots, y_d \} \subset L$ which forms a basis of $\R^d$.  
By using $Q\mathcal{B}\subset L$ 
and applying Corollary \ref{cor_sufficient_condition_for_purepoint} to 
$\mathcal{B}$, we see
the dynamical system is pure discrete.
\end{proof}

\begin{ex}
       The condition $Q(L)\subset L$ is always satisfied if $Q=nI$ for some natural number $n$.
       The table tilings (Example \ref{ex_table-substi}) do not admit full-rank infinite arithmetic progressions.
\end{ex}

\subsection{Pure discrete spectrum and a 
lattice property imply full-rank infinite arithmetic progressions}
\label{subsection_pure-discrete_implies_full-rank-AP}

In this subsection, we discuss the existence of full-rank infinite arithmetic progressions
under the following setting:
\begin{setting}\label{setting_existence_fullrankAP}
Let $\omega$ be a primitive FLC substitution rule in $\R^d$ with an expansion map $Q$.
Let us denote the alphabet by $\calA=\{T_1,T_2,\ldots, T_{m}\}$.
Let $\calT$ be a self-affine tiling generated by $\omega$.
We take a control point from each tile of $\calT$. 
The symbol $\Lambda_i$ will denote the set of all control points of tiles of type $i$.
Let $\Xi=\Xi(\calT)$ be the set of all return vectors of $\calT$.

\end{setting}

%
In this subsection, we first
 prove sufficient conditions for $\calT$ to admit full-rank infinite arithmetic
progressions.
We use the following characterization of having pure discrete dynamical spectrum:
\begin{thm}[\cite{Lee_inter_model_set}, Theorem 3.13]\label{thm_alg_coin_iff_pp}
     Let $\calS$ be a self-affine tiling with an expansion map $Q$ and an
     alphabet $\mathcal{B}$.
     We take control points for the tiles in $\calS$, and let
     $M_{T}$ be the set of all control points of type $T\in\mathcal{B}$.
     Then the following are
     equivalent:
     \begin{enumerate}
      \item $\calS$ has pure discrete dynamical spectrum;
      \item There are an $N>0$, an $S\in\mathcal{B}$ and an $\eta\in M_S$ such that
	     \begin{align*}
	           Q^N(\Xi(\calS))\subset M_S-\eta,
	     \end{align*}
	       where $\Xi(\calS)$ is the set of all return vectors for $\calS$.
     \end{enumerate}
\end{thm}

The second condition in this characterization is called the \emph{algebraic coincidence}.
Using algebraic coincidences, we now prove the following.


\begin{cor}\label{thm_alg-coin_equiv_Xi-includes-lattice}
       Let $L$ be a discrete subgroup of $\R^{d}$. Consider the following four conditions:
       \begin{enumerate}
              \item there are a positive integer $K$ and a $T\in\calT$ such that
                $\{T+x\mid x\in Q^{K}(L)\}\subset\calT$;
       \item $\Xi\supset Q^{K}(L)$ for some positive integer $K$;
       \item there is a positive integer $K$ such that, for each positive integer $k$, we have
                \begin{align}
                          \underbrace{\Xi+\Xi+\cdots +\Xi}_{\text{$k$ times}}\supset Q^{K}(L);
                          \label{eq_XiK_includes_L}
                \end{align} 
        \item there are positive integers $k$ and $K$ such that \eqref{eq_XiK_includes_L} holds.
       \end{enumerate}
       Then we always have $1.\Rightarrow 2.\Rightarrow 3.\Rightarrow 4.$, and if $\calT$ have
       pure discrete dynamical spectrum, we have $4.\Rightarrow 1.$
\end{cor}
\begin{proof}
       The implication $1.\Rightarrow 2.$ is proved by the definition of $\Xi$,
       the implication $2.\Rightarrow 3.$ is seen by the fact that $0\in\Xi$ and
       the implication $3.\Rightarrow 4.$ is clear. Let us now assume the condition $4.$ and
       that $\calT$ has pure discrete dynamical spectrum.
       
       By the algebraic coincidence, we have
       \begin{align*}
              Q^{2N}(\Xi+\Xi)\subset Q^{N}(\Lambda_{i}-\Lambda_{i})\subset Q^{N}(\Xi).
       \end{align*}
       By induction, we see
       \begin{align*}
               Q^{kN+K}(L)\subset
               Q^{kN}(\underbrace{\Xi+\Xi+\cdots +\Xi}_{\text{$k$ times}})\subset Q^{N}(\Xi)
               \subset\Lambda_{i}-\eta.
       \end{align*}
       This implies that $\{T_{i}+\eta+x\mid x\in Q^{kN+K}(L)\}\subset\calT$.
\end{proof}

By this theorem, if $\calT$ has pure discrete dynamical spectrum,
 the question of whether $\calT$ admits arithmetic progressions boils down to
\textcolor{red}{deciding}
 if the conditions 2.,3., or 4. hold. We first discuss the easiest case where these
conditions are satisfied.
\begin{cor}\label{cor_case_Xi_generates_lattice}
       Assume that the group $L$ generated by $\Xi$ is a lattice in $\R^d$ and that the tiling $\calT$ 
       has pure discrete dynamical spectrum. Then there
       are a positive integer $K$ and $T\in\calT$ such that
                $\{T+x\mid x\in Q^{K}(L)\}\subset\calT$.
                In other words, $\calT$ admits a full-rank infinite arithmetic progression.
\end{cor}

\begin{proof}

    Since $\Xi$ is relatively dense in $\R^d$,
       we may take an $R>0$ such that $\Xi+B_R=\R^d$.
       The set $L\cap B_R$ is a finite set and so there is a positive integer $M$ such that
       \begin{align*}
                L\cap B_R\subset\underbrace{\Xi+\Xi+\cdots +\Xi}_{\text{$M$ times}}.
       \end{align*}
       For each $x\in L$, there is a $y\in\Xi$ such that $\|x-y\|\leq R$. Since $x-y\in L\cap B_R$, we see
       \begin{align*}
              x=y+x-y\in\underbrace{\Xi+\Xi+\cdots +\Xi}_{\text{$M+1$ times}}.
       \end{align*}
       We conclude that
       \begin{align*}
               L=\underbrace{\Xi+\Xi+\cdots +\Xi}_{\text{$M+1$ times}}.
       \end{align*}
       The equivalent conditions in Corollary \ref{thm_alg-coin_equiv_Xi-includes-lattice} are
       satisfied.
\end{proof}

\begin{ex}\label{ex_block_substi}
      Suppose $\omega$ is a block substitution (a 
      block inflation in \cite[Subsection 5.2]{Manibo_thesis}).
      This is a substitution rule where the supports of the proto-tiles are all $[0,1]^d$ and 
      the map $Q$ is given by a diagonal matrix with integer diagonal elements.
      The group generated by $\Xi$ is always a lattice.
      $\calT$ admits full-rank infinite arithmetic progressions if it has pure discrete
      dynamical spectrum.
\end{ex}

We next give a general sufficient condition for the condition 4. in Corollary
\ref{thm_alg-coin_equiv_Xi-includes-lattice}. The proof is given in Appendix.

\begin{prop}\label{prop_sum-of-Xi_includes_lattice}
      Let $1\leqq e\leqq d$ and $v_{1}, v_{2},\ldots ,v_{e}$ be vectors in $\R^{d}$.
      Assume the following conditions:
      \begin{enumerate}
      \item for each $i=1,2,\ldots ,e$, there is a positive integer $n_{i}$ such that
                $Qv_{i}=n_{i}v_{i}$, and
        \item there are $T^{(0)}\in\calT, T^{(0)}_{1},T^{(0)}_{2},\ldots ,T^{(0)}_{e}\in\omega(T^{(0)})$
         of the same tile type such that
               \begin{align*}
                       c(T^{(0)}_{i})- Q(c(T^{(0)}))=v_{i}
               \end{align*}
               for each $i$.
      \end{enumerate}
      Then there is a positive integer $M$ such that we have
      \begin{align}
              \spa_{\Z}\{v_{1},v_{2},\ldots ,v_{e}\}\subset
              \underbrace{\Xi+\Xi+\cdots +\Xi}_{\text{$M$ times}}.
      \end{align}
\end{prop}

\begin{cor}\label{cor_general_sufficient_condition_existence_full-rankAP}
      Let $e\leqq d$ and $v_{1},v_{2},\ldots ,v_{e}$ be vectors in $\R^{d}$ that form a 
      linearly independent set. Denote the group generated by $\{v_{1},v_{2},\ldots ,v_{e}\}$ by
      $L$. If the two assumptions in Proposition \ref{prop_sum-of-Xi_includes_lattice} is satisfied
      and $\calT$ has pure discrete dynamical spectrum, then there is a tile $T\in\calT$ and
      a positive integer $K$ such that $\{T+x\mid x\in Q^{K}(L)\}\subset\calT$.
\end{cor}
\begin{proof}
      The claim follows from Corollary \ref{thm_alg-coin_equiv_Xi-includes-lattice} and
      Proposition \ref{prop_sum-of-Xi_includes_lattice}.
\end{proof}

For the self-similar case, the situation is simple:
\begin{thm}\label{thm_self-similar_inf-fullrank-AP}
     Assume $Q=nI$, where $n$ is a positive integer and $I$ is the identity map.
     Then there are a lattice $L$ of $\R^{d}$ and a positive integer $M$ such that
     \begin{align*}
           L\subset\underbrace{\Xi+\Xi+\cdots+\Xi}_{\text{$M$ times}}.
     \end{align*}
     If $\calT$ has pure discrete dynamical spectrum, then
     there are also a positive integer $K$  and a $T\in\calT$ such that
     \begin{align*}
           \{T+x\mid x\in n^{K}L\}\subset\calT.
     \end{align*}
     In other words, $\calT$ admits a full-rank infinite arithmetic progression.
\end{thm}
\begin{proof}
       We give two proofs and the first one relies on Proposition
       \ref{prop_sum-of-Xi_includes_lattice}.
       By replacing $Q$ and $\omega$ with their powers, by repetitivity we may take a
       $T^{(0)}\in\calT$, a linearly independent $\{v_{1},v_{2},\ldots ,v_{d}\}\subset\R^{d}$ and
       $T^{(0)}_{1},T^{(0)}_{2},\ldots T^{(0)}_{d}\in\omega(T^{(0)})$ such that
       \begin{itemize}
       \item $T^{(0)}_{1},T^{(0)}_{2},\ldots T^{(0)}_{d}$ are the same tile type as $T^{(0)}$, and
       \item $c(T^{(0)}_{i})-Q(c(T^{(0)}))=v_{i}$ for each $i$.
       \end{itemize}
       The first statement of this theorem is proved by Proposition \ref{prop_sum-of-Xi_includes_lattice}.
       The second statement is seen by Corollary \ref{thm_alg-coin_equiv_Xi-includes-lattice}.
       
       The second proof does not use Proposition
       \ref{prop_sum-of-Xi_includes_lattice} and relies on a result by Kenyon, which is found in
       \cite[Theorem 5.1]{Solomyak_eigenfunction}, or its generalization \cite[Theorem 4.1]{LeeSol:08}.
       There is a basis $\{b_1,b_2,\ldots ,b_d\}$ of $\R^d$ such that
       \begin{align*}
               \Xi\subset\mathbb{Z}b_1+\mathbb{Z}b_2+\cdots +\mathbb{Z}b_d.
       \end{align*}
       The group $L$ generated by $\Xi$ is a lattice of $\R^d$.        
       By the same argument as Corollary \ref{cor_case_Xi_generates_lattice},
       the first statement is proved.
       The second statement follows from
       Corollary \ref{thm_alg-coin_equiv_Xi-includes-lattice}.
\end{proof}

We will show that for irrational expansion factors, there are no full-rank infinite arithmetic
progressions in self-similar tilings
(Theorem \ref{thm_complete_picture_for_self-similar}).

In fact, we can prove a stronger statement that $\calT$ is limit-periodic, which is defined as follows.
To define it, recall the densities of patches are defined in \eqref{def_density} in page \pageref{def_density}.
Using densities, we define the limit periodic tilings:
\begin{defi}[\cite{Dirk}, Definition 4.2]
      We say that
      a tiling $\calS$ is limit periodic if there are a subset $\calP\subset\calS$ of
      zero density and a decreasing sequence $L_{1}\supset L_{2}\supset\cdots$ of lattices of $\R^{d}$
      such that for any $T\in\calS\setminus\calP$ there is $n$ with
      $\{T+x\mid x\in L_{n}\}\subset\calS$.
 \end{defi}

For example, the period doubling sequence and the repetitive fixed points of the chair substitution
 are limit periodic (\cite{BMS-limit_periodic}).
 
\begin{rem}
      It is hard to find the definition of limit-periodic tilings in the literature and different authors
      seem to use the term for different meanings.
      Limit-periodic tilings are geometric analogues for Toeplitz sequences \cite{JK}.
      The regular Toeplitz sequences have pure discrete dynamical spectrum, but
      there are non-regular Toeplitz sequences with non-zero continuous spectrum
      \cite{IL}, which means a ``Toeplitz structure'' in a definition of limit-periodic tilings
      does not imply pure discrete dynamical spectrum.
     The definition of limit-periodicity in \cite{Baake-Grimm_limit-periodic} seems to
     include pure discrete dynamical spectrum, whereas in \cite{GK} the authors do
     not include it.
     Here, we adopt the definition in \cite{Dirk} and do not assume pure discrete spectrum.
     Note that by
     Theorem \ref{full-rank-infiniteAP_implies_pp}, if $\calT$ is self-similar
     with $Q=nI$, where $n$ is a positive 
     integer and $I$ is the identity map, the limit periodicity of $\calT$ implies the pure 
     discrete dynamical spectrum.
\end{rem}

We now have the following theorem.
Note that we assumed the conditions in
Setting \ref{setting_existence_fullrankAP} at the beginning of this subsection.

\begin{thm}\label{thm_limit_periodic}
      If $\calT$ has pure discrete dynamical spectrum and
      satisfies the equivalent conditions in Corollary \ref{thm_alg-coin_equiv_Xi-includes-lattice},
      then the tiling $\calT$ is limit periodic.
      In particular, there is a lattice $L$ and a
      patch $\calP\subset\calT$ of zero density such that,
       if $T\in\calT\setminus\calP$, then $\{T+x\mid x\in Q^{l}(L)\}\subset\calT$ 
       for some natural number $l$.
\end{thm}
\begin{proof}
       The proof is given in Appendix, and here, we shall give an intuitive explanation of the idea of the
       proof. We recall the set $D_{x}=\calT\cap (\calT-x)$ in \eqref{def_Dx}. For each $k=1,2,\ldots$, 
       consider the set
       \begin{align}
              \bigcap_{x\in\Xi}D_{Q^{k}x}.\label{eq_bigcapDx}
       \end{align}
       If $T$ is in this set, then for all $x\in\Xi$, we have $T\in\calT-Q^{k}(x)$.
       By assumption, there is an $M>0$ such that $Q^{M}(L)\subset\Xi$, and so
       for each $x\in L$, we have $T+Q^{k+M}(x)\in\calT$. In other words, $T$ is a part of a full-rank
       infinite arithmetic progression. If the density of the set \eqref{eq_bigcapDx} is close to $1$,
        then there are ``many'' $T$ in this set, and each of such $T$ 
        satisfies the condition $\{T+Q^{M+k}(x)\mid x\in L\}\subset\calT$.
        Thus it suffices to show that the set \eqref{eq_bigcapDx} has large density.
        More precisely, it suffices to show the density of \eqref{eq_bigcapDx} tends to $1$ as
        $k\rightarrow\infty$.
        
        This is seen by modifying the proof for the claim that an overlap coincidence implies
        an algebraic coincidence \cite{Lee_inter_model_set}. The set \eqref{eq_bigcapDx} is written as
        \begin{align*}
              \bigcap_{x\in\Xi}D_{Q^{k}(x)}=\bigcup_{T\in\calT}\bigcap_{x\in\Xi}\omega^{k}(T)\cap
              \omega^{k}(\calT-x).
        \end{align*}
        In other words, for
        each $T\in\calT$, we can consider all possible overlaps
        $(T,x,S)$. By inflation, it gives rise to coincidences and the set of all coincidences for this overlap
        \textcolor{red}{by the $k$th inflation}
        is $\omega^k(T)\cap\omega^k(S-x)$.
        Inside the supertile $\omega^k(T)$, there are other sets of coincidences
        $\omega^k(T)\cap\omega^k(S'-x)$ with the same $x$ but a different tile $S'$. Fixing $x$,
        we collect
        all such coincidences, and the resulting set is $\omega^k(T)\cap\omega^k(\calT-x)$.
        The ratio of the area that this set of coincidences covers to the area that the supertile
        $\omega^k(T)$ covers tends to $1$, by the argument of Lemma \ref{equivalence-to-coincidence}.
        Moreover, since there are only finitely many overlaps, the intersection
       \begin{align}
                \bigcap_{x\in\Xi}\omega^{k}(T)\cap
              \omega^{k}(\calT-x)
              \label{eq_intersection_overlaps_for_T}
       \end{align}
       is in fact a finite intersection.
       For each $x\in\Xi$, the set $\omega^{k}(T)\cap \omega^{k}(\calT-x)$
       grows, and so the ratio of the area that the patch
       \eqref{eq_intersection_overlaps_for_T} covers to the support $\supp\omega^{k}(T)$ tends to $1$.
       
       This happens for each $T\in\calT$, and again by using the fact that there are
       only finitely many overlaps up to equivalence, the convergence of the ratio of the area of
       \eqref{eq_intersection_overlaps_for_T} to $1$ is uniform for all
       $T\in\calT$. This in turn implies that the density of \eqref{eq_bigcapDx} tends to $1$.
\end{proof}

\begin{cor}\label{cor_sufficient_conditions_for_limit-period}
       Suppose that the tiling $\calT$ has pure discrete dynamical spectrum.
      Moreover, assume one of the following conditions:
      \begin{enumerate}
      \item the group generated by $\Xi$ is a lattice;
      \item there is a basis $\{b_1,b_2,\ldots ,b_d\}$ of $\R^d$ that satisfies the two
      assumptions in Proposition \ref{prop_sum-of-Xi_includes_lattice};
      \item there is a positive integer $n$ such that $Q=nI$.
      \end{enumerate}
      Then $\calT$ is limit-periodic.
\end{cor}
\begin{proof}
       By using Theorem \ref{thm_limit_periodic},
       this follows from Corollary \ref{cor_case_Xi_generates_lattice},
       Proposition \ref{prop_sum-of-Xi_includes_lattice} and
        Theorem \ref{thm_self-similar_inf-fullrank-AP}.
\end{proof}

\begin{ex}
       The block substitution (Example \ref{ex_block_substi}) satisfies the first condition in
       Corollary \ref{cor_sufficient_conditions_for_limit-period}. If the substitution satisfies
       the coincidence condition (a high-dimensional version of Dekking's condition
       \cite[Definition 6.5]{Q}, which follows from the overlap algorithm 
       \cite{Solomyak_Tiling}),
       then it is limit-periodic.

       The self-similar tiling by the 
       sphinx substitution \cite{Godreche-sphinx} are known to be limit-periodic.
       We give another proof for this fact: we can check that the sphinx substitution
       satisfies the overlap coincidence \cite[Example 7.2]{Solomyak_Tiling};
       by Corollary \ref{cor_sufficient_conditions_for_limit-period},
       the fact that the expansion factor is $2$ implies the sphinx tilings are
       limit-periodic.
\end{ex}

\section{The conclusion for full-rank infinite arithmetic progressions in self-similar tilings}\label{section_one-dim}

By what we have proved, we have the following complete picture for the 
existence/non-existence of full-rank infinite arithmetic progressions in any self-similar tiling:

\begin{thm}\label{thm_complete_picture_for_self-similar}
      Let $\mathcal{T}$ be a self-similar tiling in $\mathbb{R}^d$ with a
      primitive FLC substitution and an expansion factor $\lambda$.
      \begin{enumerate}
            \item If $\lambda$ is irrational, then there are no full-rank infinite arithmetic progressions in 
                      $\mathcal{T}$.
            \item If $\lambda$ is rational (i.e. it is a natural number), then the following three
            conditions are equivalent:
            \begin{enumerate}
            \item   $\mathcal{T}$ admits full-rank infinite
                     arithmetic progressions;
            \item $\mathcal{T}$ is limit periodic;
            \item $\mathcal{T}$ has pure discrete dynamical spectrum.
            \end{enumerate}
       \end{enumerate}
\end{thm}

\begin{proof}
       The second statement is clear by Theorem \ref{full-rank-infiniteAP_implies_pp} and
        Corollary \ref{cor_sufficient_conditions_for_limit-period}.
        We shall prove the first statement.
        
        Suppose that $\lambda$ is irrational.
        The statement follows from the observation that if we have two full-rank
        infinite arithmetic progressions with incommensurate distances, then there are
        two different tiles that are arbitrarily close,
        which contradicts the  fact that $\calT$ is a tiling.
        We shall elaborate the argument.
         We assume there are a
        $T\in\mathcal{T}$ and a lattice $L$ of $\R^d$ such that
        $\{T+x\mid x\in L\}\subset\calT$ and we will obtain a contradiction.
        We take a basis $\{b_1,b_2,\ldots ,b_d\}$ of $\R^d$
        such that $L$ is the $\mathbb{Z}$-span of this basis.
        
        We can take a natural number $N$ such that 
        there is a $y\in\R^d$ with $T+y\in\omega^N(T)$. For each $\varepsilon>0$, we can take
        integers $n_j,m_j$ ($j=1,2,\ldots, d$), such that
        \begin{align*}
              0<\|\lambda^N\sum_{j=1}^dn_jb_j+y-\sum_{j=1}^dm_jb_j\|<\e.
        \end{align*}
        These inequalities show that the two tiles $T+\lambda^N\sum n_jb_j+y$ and
        $T+\sum m_jb_j$ are different but their interiors overlap.
        Since these tiles are both in $\calT$, we obtain a contradiction.
\end{proof}

\begin{rem}
      Although we described Theorem \ref{thm_complete_picture_for_self-similar} as a
      ``complete picture'', it is not clear when $\mathcal{T}$ has pure discrete dynamical spectrum.
      The Pisot conjecture, which states a self-similar tiling for any irreducible substitution
      with Pisot expansion factor has pure discrete dynamical spectrum, has not been proved.
      However, we have an algorithm to decide if a given self-similar tiling has pure discrete dynamical
      spectrum (\cite{AL_algorithm,Solomyak_Tiling}).
\end{rem}

For the one-dimensional case,
if the substitution $\omega$ is also irreducible, we can prove the first claim of 
Theorem \ref{thm_complete_picture_for_self-similar} by using the Perron-Frobenius theory, as
follows.

Assume that $\omega$ has an
 indexed alphabet $\calA=\{T_{1},T_{2},\ldots ,T_{m}\}$ and is primitive.
 Assume also that the support of each tile in $\calA$ is an interval.
By the Perron-Frobenius theorem, there is a positive eigenvalue that is
greater than any other eigenvalues in modulus. We call it the Perron-Frobenius eigenvalue, and
for substitutions it coincides with the expansion factor. Let $(l_{j})_{j=1,2,\ldots, m}$ and 
$(r_{j})_{j=1,2,\ldots, m}$ be left and right eigenvectors, respectively, with positive elements, which
exist uniquely up to scalar multiples.

\begin{lem}\cite[Lemma 4.3]{OverlapStrong}
\label{Linindep}
If $\omega$ is primitive and irreducible, then the $d$-entries of
$(l_j)_{j\in A}$ (resp. $(r_j)_{j\in A}$) are linearly independent over $\Q$.
\end{lem}

Recall that by primitivity, the corresponding tiling dynamical system is uniquely ergodic, which is
equivalent to having uniform patch frequency.

We now prove that an 
infinite arithmetic progression cannot be observed in a self-similar tiling
generated by an irreducible primitive substitution $\omega$ with the cardinality of the alphabet
greater than $1$.
Note that the irreducibility implies that $\lambda$ is irrational.

\begin{thm}
\label{IrrPrm}
Let $\omega$ be a one-dimensional
 irreducible primitive substitution with interval supports of the tiles in the alphabet
 and with the cardinality of the alphabet greater than 1. 
 Let $\Tk_{0}$ be a repetitive fixed point for $\omega$.
If $\calT$ is a tiling in $X_{\calT_{0}}$,
then for any tile $T\in \Tk$ and
 non-zero vector $v$,  the patch $\{T+kv\mid k\in\N\}$ is not included in $\Tk$.
 Therefore, there are no arbitrarily long arithmetic progressions if we fix the distance $v\neq 0$.
\end{thm}

\begin{proof}
Assume that $\{ T+kv\mid k\in\N\}$ is  included in $\Tk$.
We prove that this leads to a contradiction, for the case where $v>0$.
The case where $v<0$ is proved in a similar way.
Consider a legal patch
$\calP_k\ (k\in \N)$ of interval support
whose left most tile is $T+kv$ and right most tile is the tile just before $T+(k+1)v$.
The support of the patch $\calP_k$ is an interval of the fixed length $v$ which is filled by translates of alphabets $T_j$'s with cardinalities 
$n_j$'s where $j\in \{1,2,\ldots n\}$.
Then the cardinality vector $(n_j)_{j\in \{1,2,\ldots n\}}\in \N^d$ is 
independent of the choice of $k$ by Lemma \ref{Linindep}. 
For each $m>0$, the patch $\bigcup_{k=0}^{m-1} \calP_{k}$ contains $m n_j$
translates of $T_j$ for $j\in \{1,2,\ldots ,n\}$ and has support $[0,mv]$.
By the uniform patch frequency, this shows that $r_i=n_i/(\sum_{j=1}^{n} n_j)\in \Q$ for 
$i\in \{1,2,\ldots,n\}$.
It gives a contradiction to Lemma \ref{Linindep}.
\end{proof}

\section{Further questions}
\label{sec_further_questions}
In Theorem \ref{thm_using_conti_eigenft}, we assumed that
the expansion factor $\lambda$ is
irrational and Pisot.
If $\lambda$ is rational, which means it is a natural number, there may be arbitrarily long arithmetic
progressions even if we fix the distance $x\in\R^{d}\setminus\{0\}$, as we have seen in
Example \ref{ex_chair_infinite-AP} and Example \ref{ex_table-substi}.
However, we do not understand the situation where we drop the Pisot assumption.
We do not know any examples with irrational non-Pisot expansion factors that admit one-dimensional arithmetic progressions of arbitrary length.
Moreover, the existence and absence of \emph{infinite one-dimensional} arithmetic progressions
may depend on the choice of a tiling from the hull. For example, we do not know if there is a tiling
in the hull of the table substitution that does not admit infinite one-dimensional arithmetic progressions.

The study of full-rank infinite arithmetic progressions in Section \ref{section_full-rank-AP} is  also not
complete. 
We understand the self-similar case, but for general $Q$, there are still open problems.
For example, we suspect if all the eigenvalues for $Q$ are integers and the
tiling has pure discrete dynamical spectrum, then the tiling is limit-periodic, but we do not have
a proof for this.
We have assumed several conditions, such as that 
$Q(L)\subset L$ in Theorem \ref{full-rank-infiniteAP_implies_pp}.
We do not know what happens if we drop these assumptions.

\section{Appendix}
Here, we put some proofs in the main part of the paper.
\subsection{The proof for Proposition \ref{prop_sum-of-Xi_includes_lattice}}
First we prove  Proposition \ref{prop_sum-of-Xi_includes_lattice}.
\begin{proof}
     We set $M=e\max\{n_{1},n_{2},\ldots ,n_{e}\}$ and prove that this $M$ satisfies the desired 
     condition.
     To prove this,  take arbitrary $k_{1},k_{2},\ldots ,k_{e}\in\Z$ and we shall prove that
     $\sum_{i=1}^{e}k_{i}v_{i}$ is included in $\Xi+\Xi+\cdots +\Xi$ (the result of the summation for
     $M$ $\Xi$'s).
     We use a tile map $f\colon\calT\rightarrow\calT$ which was used to define a control points map
     $c$. We can assume that $T$ and $f(T)$ are always the same tile type.
     For any $i=1,2,\ldots ,e$ such that $k_{i}\geqq 0$, by considering $n_{i}$-adic
     expansion of $k_{i}$, there are
     $k_{i,0},k_{i,1},\ldots ,k_{i,N_{i}}\in\{0,1,\ldots n_{i}-1\}$ such that
     \begin{align*}
             k_{i}=\sum_{j=0}^{N_{i}}k_{i,j}n_{i}^{N_{i}-j}.
     \end{align*}
     For $i$ with $k_{i}<0$, similarly there are $k_{i,0},k_{i,1},\ldots,k_{i,N_{i}}\in\{0,1,\ldots ,n_{i}\}$
     such that
          \begin{align*}
             -k_{i}=\sum_{j=0}^{N_{i}}k_{i,j}n_{i}^{N_{i}-j}.
     \end{align*}
     Since each $N_{i}$ can be replaced with a larger integer, we may assume that
     $N_{1}=N_{2}=\cdots N_{e}$. Let us denote this number by $N$.
     Set $I_{+}=\{i=1,2,\ldots ,e\mid k_{i}\geqq 0\}$ and $I_{-}=\{i=1,2,\ldots ,e\mid k_{i}< 0\}$.
     
     For each $i\in I_+$ and $l=1,2,\ldots,n_{i}-1$,
     we shall choose tiles 
     \begin{align*}
      S_{N+1,l,i}, S'_{0,l,i}, 
     \end{align*}
     and for each $i\in I_-$ and $l=1,2,\ldots ,n_i$, we shall chose tiles
     \begin{align*}
         T_{N+1,l,i}, T'_{0,l,i},
     \end{align*}
      all in $\calT$ and with the same tile type,
     such that
          \begin{align}
           \sum_{i\in I_{+}} \sum_{l=1}^{n_{i}-1}(c(S_{N+1,l,i})-c(S'_{0,l,i}))-
           \sum_{i\in I_{-}}\sum_{l=1}^{n_{i}-1}(c(T_{N+1,l,i})-c(T'_{0,l,i}))=\sum_{i=1}^{e}k_{i}v_{i}.
           \label{eq_sumupto_kivi}
     \end{align}
     Since all objects $c(S_{N+1,l,i})-c(S'_{0,l,i})$ and $c(T_{N+1,l,i})-c(T'_{0,l,i})$, where
     $i=1,2,\ldots,e$ and $l=1,2,\ldots n_{i}-1$, are members of $\Xi$, and the number of these
     objects is at most $M$, once we have proved this equation \eqref{eq_sumupto_kivi}, 
     we finish the proof of this proposition.
     
     First, take $i\in I_{+}$ and $l=1,2,\ldots n_{i}-1$,
     and we shall construct $S_{N+1,l,i}$ and $S'_{0,l,i}$.
     The latter is defined by $S'_{0,l,i}=f^{N+1}(T^{(0)})$. To construct the former,
     set $S_{0,l,i}=T^{(0)}$ and we choose
     $S_{1,l,i}, S_{2,l,i},\ldots, S_{N+1,l,i}$ inductively, 
     as follows. Given an $S_{j,l,i}$ of the same tile type as
     $T^{(0)}$, by assumption we can take an $S_{j+1,l,i}\in\omega(S_{j,l,i})$ of the same type as
     $T^{(0)}$ such that
     \begin{align*}
               c(S_{j+1,l,i})-Q(c(S_{j,l,i}))=
               \begin{cases}
                         v_{i}      &\text{if $k_{i,j}\geqq l$}\\
                         0            &\text{otherwise}.
               \end{cases}
     \end{align*}
     By induction, we have found $S_{N+1,l,i}$.
     
     Similarly,
       for each $i\in I_{-}$ and $l=1,2,\ldots n_{i}-1$,
     we can take $T_{0,l,i}=T^{(0)},T_{1,l,i},\ldots ,T_{N+1,l,i}\in\calT$ of the same type so that
          \begin{align*}
               c(T_{j+1,l,i})-Q(c(T_{j,l,i}))=
               \begin{cases}
                         v_{i}     &\text{if $k_{i,j}\geqq l$}\\
                         0            &\text{otherwise}.
               \end{cases}
     \end{align*}
     By induction, we have defined $T_{N+1,l,i}$.
     $T'_{0,l,i}$ is defined as $f^{N+1}(T_{0,l,i})$.
     
     Now we have defined the necessary tiles, and we shall prove the equation
     \eqref{eq_sumupto_kivi}.
     By using the fact that
     \begin{align*}
             \sum_{l=1}^{n_{i}-1}(c(S_{j,l,i})-Q(c(S_{j-1,l,i}))))=k_{i,j-1}v_{i},
     \end{align*}
     we have the following:
     \begin{align*}
             \sum_{l=1}^{n_{i}-1}(c(S_{N+1,l,i})-c(S'_{0,l,i}))&=
             \sum_{l=1}^{n_{i}-1}(c(S_{N+1,l,i})-Q^{N+1}(c(S_{0,l,i})))\\
             &=\sum_{j=1}^{N+1}\sum_{l=1}^{n_{i}-1}Q^{N-j+1}(c(S_{j,l,i})-Q(c(S_{j-1,l,i})))\\
             &=\sum_{j=1}^{N+1}Q^{N-j+1}(k_{i,j-1}v_{i})\\
             &=\sum_{j=1}^{N+1}n_{i}^{N-j+1}k_{i,j-1}v_{i}\\
             &=k_{i}v_{i}.
     \end{align*}
    
    Also, we have an equation
     \begin{align*}
             \sum_{l=1}^{n_{i}-1}(c(T_{N+1,l,i})-c(T'_{0,l,i}))=-k_{i}v_{i}.
     \end{align*}
     These two equations prove the equation \eqref{eq_sumupto_kivi}.
\end{proof}

\subsection{The proof for Theorem \ref{thm_limit_periodic}}
Here, we prove Theorem  \ref{thm_limit_periodic}.

Let $\omega$ be a primitive FLC 
substitution rule with an expansion map $Q$ and an alphabet $\calA$.
For each $P\in\calA$, the patch $\omega(P)$ is the result of ``expanding $P$ by $Q$ and then
subdividing it''.  Henceforth we also use a
``subdividing without expanding'' map $\sigma$.
By using this $\sigma$ instead of $\omega$, the proof is
simplified.
The map
$\sigma$ is defined via
\begin{align*}
      \sigma(P)=Q^{-1}(\omega(P)).
\end{align*}
(See Figure \ref{figure_subdivision-map}.)
Here, we set $Q^{k}(T)=(Q^{k}(S),l)$ for each labeled
tile $T=(S,l)$ and $k\in\Z$, and $Q^{k}(\calP)=\{Q^{k}(T)\mid T\in\calP\}$
for each patch $\calP$.
(For unlabeled tiles $T$, $Q^{k}(T)$ has the usual meaning.)
We will also ``subdivide without expanding'' $Q^{l}(P)+x$'s, where $l\in\Z$, $x\in\R^{d}$ 
and $P\in\calA$, and so
define
\begin{align*}
      \sigma(Q^{l}(P)+x)=Q^{l-1}(\omega(P))+x.
\end{align*}
For each patch $\calP$ of which elements are of the form $Q^{l}(P)+x$, we set
\begin{align*}
      \sigma(\calP)=\bigcup_{T\in\calP}\sigma(T).
\end{align*}

\begin{figure}[htbp]
\begin{center}
\includegraphics{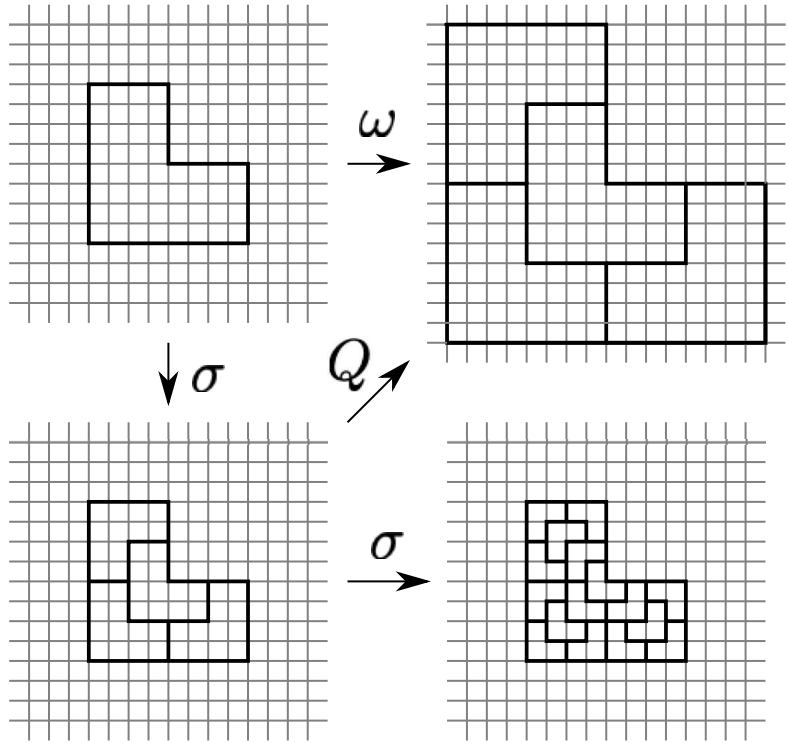}
\caption{The chair substitution $\omega$ and the corresponding subdivision map $\sigma$}
\label{figure_subdivision-map}
\label{default}
\end{center}
\end{figure}

Then we have the following properties:

\begin{enumerate}
     \item $\sigma(\calP)$ is also a patch consisting of tiles of the form $Q^{l}(P)+x$ and
              $\supp\sigma(\calP)=\supp\calP$.
     \item $\sigma\circ Q^{l}=Q^{l}\circ\sigma$ for any $l\in\Z$.
\end{enumerate}

We now assume Setting \ref{setting_existence_fullrankAP} 
in Subsection \ref{subsection_pure-discrete_implies_full-rank-AP} and
the self-affine tiling $\calT$ has pure discrete dynamical spectrum and satisfies the equivalent
conditions in Corollary \ref{thm_alg-coin_equiv_Xi-includes-lattice}.
To prove Theorem \ref{thm_limit_periodic},
we divide the proof into lemmas. First, by using Lemma \ref{equivclass_overlap_finite} and
replacing $\omega$ with its power,  without loss of generality we may assume that,
for each overlap $(S,x,T)$ for $\calT$, we have
\begin{align}
      \omega(S+x)\cap\omega(T)\neq\emptyset,
\end{align}
which is equivalent to
\begin{align}
      \sigma(S+x)\cap\sigma(T)\neq\emptyset.
\end{align}

Set $V_{\max}=\max_{P\in\calA}\vol(P)$, $V_{\min}=\min_{P\in\calA}\vol(P)$ and
$D=|\det Q|$.
(Here, $\vol$ denotes the Lebesgue measure.)

In what follows we tacitly use the following result on the boundaries $\partial T$ of the 
tiles $T$ in
self-affine tilings:

\begin{lem}[\cite{prag}]
       Let $T$ be a tile in a self-affine tiling. Then $\vol(\partial T)=0$, where $\vol$ denotes
       the Lebesgue measure and $\partial$ denotes the boundary.
\end{lem}

\begin{lem}\label{lem_for_limit-periodic}
     For an overlap $(S_{0},y_{0},T_{0})$ and $k\in\Zpo$, set
     \begin{align*}
             K_{k}=\supp(\sigma^{k}(S_{0}+y_{0})\cap\sigma^{k}(T_{0})).
     \end{align*}

     Then, for each $k$, $K_{k}\subset K_{k+1}$ and
       \begin{align}
              1-\frac{\vol(K_{k+1})}{\vol((\supp S_{0}+y_{0})\cap\supp T_{0})}
                   \leqq \left(1-\frac{V_{\min}}{DV_{\max}}\right)
                            \left(1-\frac{\vol(K_{k})}{\vol((\supp S_{0}+y_{0})\cap \supp T_{0})}\right),
     \end{align}
     and so $\vol(K_{k})\rightarrow\vol((\supp S_{0}+y_{0})\cap \supp T_{0})$ as $k\rightarrow\infty$.
\end{lem}

\begin{proof}
    If $k>0$, $S\in\sigma^{k}(S_{0}+y_{0})$, $T\in\sigma^{k}(T_{0})$,
     $\Int S\cap\Int T\neq\emptyset$ and $S\neq T$, then there is a 
     $U_{S,T}\in\sigma(S)\cap\sigma(T)$.
     It is clear that $\Int U_{S,T}\cap K_{k}=\emptyset$, since otherwise $S=T$. We have
     \begin{align*}
            1-\frac{\vol(K_{k+1})}{\vol((\supp S_{0}+y_{0})\cap\supp T_{0})}
            \leqq 1-\frac{\vol(K_{k})}{\vol((\supp S_{0}+y_{0})\cap\supp T_{0})}\\
            -\frac{1}{\vol((\supp S_{0}+y_{0})\cap\supp T_{0})}\sum\vol(\supp U_{S,T}),
     \end{align*}
     where the sum is for all $S$ and $T$ with the above condition.
     The claim of the lemma follows from the observations
     \begin{align*}
            D^{-k}V_{\max}\card\{(S,T)\in\sigma^{k}(S_{0}+y_{0})\cap\sigma^{k}(T_{0})\mid
                    T\neq S,\Int T\cap\Int S\neq\emptyset\}\\\geqq\vol((\supp(S_{0}+y_{0})\cap\supp T_{0})\setminus K_{k}),
     \end{align*}
     and
     \begin{align*}
             \vol(\supp U_{S,T})\geqq D^{-k-1}V_{\min}.
     \end{align*}
\end{proof}

       Set 
        \begin{align*}
                 \calP_{k}=\bigcap_{x\in\Xi}\sigma^{k}(\calT-x).
        \end{align*}
        
        Note that $Q^k(\calP_k)=\bigcap_{x\in\Xi}D_{Q^k(x)}$. 
        We will show that (1)any $T\in Q^k(\calP_k)$ is a part of a full-rank
        infinite arithmetic progression, and (2)the density of $\calP_k$, which is same as 
        the density of $Q^k(\calP_k)$, tends to $1$ as $k\rightarrow\infty$.
        
        Note that by Lemma \ref{lem_for_limit-periodic}, we have
        $\supp\calP_{k}\subset\supp\calP_{k+1}$ for each $k>0$.
        
        \begin{lem}\label{lem_tiles_in_calPk_periodic}
               There is a natural number $M$ such that for any $k>0$ and $S\in\calP_{k}$,
               the tile $Q^{k}S$ is a  tile in $\calT$ with
               $\{Q^{k}S+x\mid x\in Q^{k+M}L\}\subset\calT$.
        \end{lem}
        \begin{proof}
              By assumption, there is an $M$ such that
              \begin{align}
                    Q^{M}L\subset\Xi, \label{eq_Q^ML_subset_Xi}
              \end{align}
             If $k>0$ and $S\in\calP_{k}$, then
              for each $x\in\Xi$ we have $S+x\in\sigma^{k}(\calT)$.
              By the definition of $\sigma$, we see $Q^{k}(S)+Q^{k}(x)\in Q^{k}\sigma^{k}(\calT)
              =\omega^{k}(\calT)=\calT$.
              By \eqref{eq_Q^ML_subset_Xi}, we can set $x=Q^{M}(y)$ for each $y\in L$.
        \end{proof}
        
        \begin{lem}
             If $T\in\calT$ does not satisfy  $\{T+x\mid x\in Q^{l}(L)\}\subset\calT$
              for any $l>0$, then
             $T\notin Q^{k}\calP_{k}$ for any $k$.
        \end{lem}
        \begin{proof}
              This is clear by Lemma \ref{lem_tiles_in_calPk_periodic}.
        \end{proof}
        
        By this lemma, the proof of Theorem  \ref{thm_limit_periodic} is completed if we prove the
        density of each $\calP_{k}$ is large enough. We now show this.

          \begin{lem}
                   For all van Hove sequence $(A_{n})_{n}$,  we have
                   \begin{align*}
                       \lim_{k\rightarrow\infty} \dens_{(A_{n})_{n}}\calT\setminus Q^{k}\calP_{k}= 0.
                   \end{align*}
                   (In fact, the convergence as $k\rightarrow\infty$ is uniform for all van Hove sequence
                   $(A_{n})_{n}$.)
          \end{lem}
\begin{proof}
               For each $S\in\calT$ and
                $x\in\Xi$, by Lemma \ref{lem_for_limit-periodic} we have, as $k\rightarrow\infty$,
        \begin{align}
                 \vol\supp(\sigma^{k}(S)\cap\sigma^{k}(\calT-x))
                         &=\vol\supp\left(\bigcup_{T\in\calT\sqcap (\Int S+x)}\sigma^{k}(S)\cap\sigma^{k}(T-x)\right)\nonumber\\
                         &=\sum_{T\in\calT\sqcap (\Int S+x)}\vol\supp(\sigma^{k}(S)\cap\sigma^{k}(T-x))\nonumber\\
                         &\rightarrow\sum_{T\in\calT\sqcap (\Int S+x)}\vol(\supp S\cap (\supp T-x))\nonumber\\
                         &=\vol\supp S.\label{convergence_to_volsuppS}
        \end{align}
        Since there are only finitely many equivalence classes of overlaps, 
        this convergence is uniform in $S$ and $x$.


        We use the fact that the overlaps are finite up to equivalence, for one more time.
        Take an $S\in\calT$ and fix it. Consider a map
        \begin{align*}
               \Xi\ni x\mapsto \{[S,x,T]\mid T\in\calT\sqcap (\Int S+x)\}.
        \end{align*}
        Since the range of this map is a finite set, there is a finite set $F_{S}\subset \Xi$ such that
        for any $x\in\Xi$ there is a $y\in F_{S}$ with
        \begin{align*}
                \{[S,x,T]\mid T\in\calT\sqcap (\Int S+x)\}= \{[S,y,T]\mid T\in\calT\sqcap (\Int S+y)\}.
        \end{align*}

        Note that $\card F_{S}\leqq N$  for some constant $N$ independent of $S$,
        where $N=2^{N_{0}}$ and $N_{0}$ is the number
        of all equivalence classes of overlaps in $\calT$. 
        ($N$ is the number of all subsets of the set of the equivalence classes of overlaps.)
        We will use this estimate for the estimate of
        the density of $\sigma^{k}(\calT)\setminus \calP_{k}$, as follows.

        Let $\e$ be an arbitrary positive real number. There is a natural number $k_{0}$ such that
        if $k\geqq k_{0}$, $S\in\calT$ and $x\in\Xi$, we have
        \begin{align*}
               0\leqq \vol\supp S-\vol\supp(\sigma^{k}(S)\cap\sigma^{k}(\calT-x))
                   =\vol\supp(\sigma^{k}(S)\setminus\sigma^{k}(\calT-x))<\e.
        \end{align*}
         
        For any van Hove sequence $(A_{n})_{n}$ and $k\geqq k_{0}$,
        \begin{align*}
              &\limsup_{n\rightarrow\infty} \frac{1}{\vol(A_{n})}\vol\supp((\sigma^{k}(\calT)\setminus \calP_{k})\sci A_{n})\\
                   &=\limsup_{n}\frac{1}{\vol(A_{n})}\sum_{T\in \calT\sci A_{n}}\vol\supp (\sigma^{k}(T)\setminus\calP_{k})\\
                   &\leqq\limsup_{n}\frac{1}{\vol(A_{n})}\sum_{T\in\calT\sci A_{n}}\sum_{x\in F_{T}}\vol\supp
                               (\sigma^{k}(T)\setminus\sigma^{k}(\calT-x))\\
                    &\leqq\limsup_{n}\frac{1}{\vol(A_{n})}\sum_{T\in\calT\sci A_{n}}N\e\\
                    &\leqq\frac{N\e}{V_{\min}},
        \end{align*}
        where we used the fact that
        \begin{align*}
              \sigma^{k}(T)\setminus\calP_{k}&=\bigcup_{x\in\Xi}\sigma^{k}(T)\setminus\sigma^{k}(\calT-x)\\
                             &=\bigcup_{x\in F_{T}}\sigma^{k}(T)\setminus\sigma^{k}(\calT-x)
        \end{align*}
        and
        \begin{align*}
                V_{\min}\card\calT\sci A_{n}\leqq\vol (A_{n}).
        \end{align*}

        We now finish the proof. Take a van Hove sequence $(A_{n})_{n}$. For each $k\geqq k_{0}$,
        $(Q^{-k}A_{n})_{n}$ is again a van Hove sequence. By 
        $\calT=\omega^{k}(\calT)=Q^{k}\sigma^{k}(\calT)$, we have
                \begin{align*}
               \dens_{(A_{n})_{n}}\calT\setminus Q^{k}\calP_{k}&=
                     \limsup_{n }\frac{1}{\vol(A_{n})}\vol\supp((\calT\setminus Q^{k}\calP_{k})\sci A_{n})\\
                     &=\limsup_{n}\frac{1}{\vol(Q^{-k}A_{n})}\vol\supp ((\sigma^{k}(\calT)\setminus\calP_{k})\sci 
                            Q^{-k}A_{n})\\
                     &\leqq\frac{N\e}{V_{\min}}, 
        \end{align*}
        and so $\dens_{(A_{n})_{n}}\calT\setminus Q^{k}\calP_{k}$ tends to $0$
        as $k\rightarrow\infty$.
\end{proof}

\section*{Acknowledgment }
YN was supported by EPSRC grant EP/S010335/1.
SA was supported by JSPS grants (17K05159, 17H02849, BBD30028).
JL was supported by NRF grant No. 2019R1I1A3A01060365.
This study was led by YN.
The authors thank the referee for valuable comments.
The authors also thank Dan Rust for the check of English in section 1 and 2.

\end{document}